\documentclass[twoside,11pt]{article}

\usepackage[margin=1in]{geometry}
\usepackage{mathptmx} 
\usepackage{authblk} 
\usepackage{setspace} 
\usepackage{parskip} 
\usepackage{fancyhdr} 
\usepackage{lastpage} 
\usepackage{amsmath,amssymb,amsthm}
\usepackage{bm} 
\usepackage{mathtools} 
\usepackage{physics}

\theoremstyle{plain}
\newtheorem{theorem}{Theorem}[section]
\newtheorem{lemma}[theorem]{Lemma}
\newtheorem{proposition}[theorem]{Proposition}

\theoremstyle{definition}
\newtheorem{definition}[theorem]{Definition}

\newtheorem{remark}[theorem]{Remark}

\usepackage{graphicx}
\usepackage{subcaption} 
\usepackage{booktabs} 
\usepackage{multirow} 
\usepackage{multicol} 
\usepackage{longtable} 
\usepackage{rotating} 
\usepackage{caption} 
\usepackage{float} 
\graphicspath{{figures/}} 
\usepackage[
  backend=biber,       
  sorting=nty,          
  maxnames=3,           
  minnames=1,           
  maxbibnames=5,        
  minbibnames=3,       
  doi=true,             
  isbn=false,           
  url=false,            
  eprint=false          
]{biblatex}
\DeclareFieldFormat{journaltitle}{\textit{#1}}  
\DeclareFieldFormat{booktitle}{\textit{#1}}    
\DeclareFieldFormat{title}{#1.}                

\usepackage[
  colorlinks=true,
  linkcolor=blue,
  citecolor=blue,
  urlcolor=blue,
  pdfauthor={Y. Wang},
  pdftitle={A New Algorithm for Computing the Stabilizing Solution of General Periodic Time-Varying Stochastic Game-Theoretic Riccati Differential Equations},
  pdfsubject={Iteration Algorithm, Zero-Sum Linear-Quadratic Stochastic Differential Games},
  pdfkeywords={Iteration Algorithm, Zero-Sum Linear-Quadratic Stochastic Differential Games}
]{hyperref}
\usepackage{doi} 
\usepackage{url} 

\newcommand{\keywords}[1]{\noindent\textbf{Keywords:} #1}

\newcommand{\secref}[1]{Section~\ref{#1}} 
\newcommand{\eqnref}[1]{Eq.~\eqref{#1}}

\begin{document}

\title{A New Algorithm for Computing the Stabilizing Solution of General Periodic Time-Varying Stochastic Game-Theoretic Riccati Differential Equations}
\author{Yiyuan Wang\thanks{Shandong University Zhongtai Securities Institute for Financial Studies, Shandong University, 27 Shanda Nanlu, Jinan, P.R. China, 250100 ({Email: wangyiyuan@mail.sdu.edu.cn}).}}
\date{}

\maketitle

\begin{abstract}
We propose a new algorithm for a broad class of periodic time-varying Stochastic Game-Theoretic Riccati Differential Equations arising in Zero-Sum Linear-Quadratic Stochastic Differential Games. The algorithm is constructed via dual-layer matrix-valued functions iteration sequences, which reformulate the original problem into a set of interconnected bilevel subproblems. By sequentially computing the maximal periodic solutions to the Riccati differential equations associated with each subproblem, we derive the stabilizing periodic solutions for the original problem and rigorously prove the algorithm's convergence. Numerical experiments verifies algorithm effectiveness and stability. This study provides a unified numerical framework for solving a wider range of periodic time-varying Stochastic Game-Theoretic Riccati Differential Equations.\\
\keywords{Iteration Algorithm; Stochastic Game-Theoretic Riccati Differential Equations; Stabilizing Solutions; Zero-Sum Linear-Quadratic Stochastic Differential Games} 
\end{abstract}

\section{Introduction} 
\label{sec:introduction}

We focus on the following broad class of Stochastic Game-Theoretic Riccati Differential Equations (SGTRD\\Es):  
\begin{equation}
\label{pf:gtrde}
\begin{aligned}
\frac{d}{dt}{X}(t) &+ A_0^{\top}(t)X(t) + X(t)A_0(t) + \sum_{k=1}^r A_k^{\top}(t)X(t)A_k(t) + M(t) \\
&- \left[ X(t)B_0(t) + \sum_{k=1}^r A_k^{\top}(t)X(t)B_k(t) + L(t) \right] \left[ R(t) + \sum_{k=1}^r B_k^{\top}(t)X(t)B_k(t) \right]^{-1} \\
&\quad \times \left[ B^{\top}_0(t)X(t) + \sum_{k=1}^r B_k^{\top}(t)X(t)A_k(t) + L^{\top}(t) \right] = 0,\quad t \geq 0
\end{aligned}
\end{equation}  
where the unknown function $t \mapsto X(t) \in \mathbb{S}^n$. The parameters satisfy the following properties: $A_k(t) : \mathbb{R}_+ \to \mathbb{R}^{n \times n}$, $B_k(t) : \mathbb{R}_+ \to \mathbb{R}^{n \times m}$ ($k=0,\dots,r$), $L(t) : \mathbb{R}_+ \to \mathbb{R}^{n \times m}$, $M(t) : \mathbb{R}_+ \to \mathbb{S}^n$ and $R(t) : \mathbb{R}_+ \to \mathbb{S}^m$. All these matrix-valued functions are continuous and periodic with period $\theta > 0$.

This category of SGTRDEs originates from Zero-Sum Linear-Quadratic Stochastic Differential Game problems (ZSLQSDG). Specifically, the set of admissible solutions to the SGTRDEs defined in \eqref{pf:gtrde} consists of all mappings $X(\cdot): \mathbb{R}_+ \to \mathbb{S}^n$ that satisfy the following sign condition:  
\begin{equation}
\label{pf:sign_conditions_1}
\mathrm{sgn}\!\left[ R(t) + \sum_{k=1}^r B_k^{\top}(t)X(t)B_k(t) \right] = \mathrm{sgn} \, \mathrm{diag}\left(-I_{m_1}, I_{m_2}\right),m_1+m_2=m,\, \quad \forall \, t \in \mathbb{R}_{+}.    
\end{equation}  
A detailed description of ZSLQSDG problems will be provided later in this paper; for additional background, refer to \cite{Sun2020book}, \cite{Bacsar1998book}, and the references therein.   

Under the infinite-horizon setting, the stabilizing solution of SGTRDEs is a central topic in the study of control theory, game theory and Riccati differential equations. As the global solution to the equation, it enables the derivation of stable feedback gains for the system—critical for ensuring both stable operation and equilibrium performance over an infinite time horizon. \cite{Dragan2020} studies a class of coupled nonlinear matrix differential equations derived from the ZSLQSDG problems. The underlying dynamical system is governed by an It\^o equation with randomly switching coefficients. The paper provide sufficient conditions for the existence of bounded and stabilizing solutions to these Riccati equations. To differentiate this type of equation from $H_\infty$-type SGTRDEs (which serve as their subproblems), we term them ZSLQSDG-type SGTRDEs throughout this paper.

Notably, the stabilizing solution of a Riccati differential equation cannot be defined as either an initial-value problem (IVP) or a boundary-value problem (BVP) solution. This inherent property renders all existing numerical methods for solving IVPs or BVPs inapplicable to computing stabilizing solutions of Riccati differential equations. Furthermore, Kleinman-type iterative methods in \cite{Kleinman1968}—originally designed for problems with definite-sign quadratic terms—fail to extend to Riccati differential equations with indefinite-sign quadratic terms. For the theoretical and numerical results regarding the Riccati equation with a definite quadratic term, please refer to \cite{gajic1999,guo2001,freiling2003,damm2004,abou2012matrix,Dragan2013book} and its citations.

In the literature, \cite{Lanzon2008} proposed an iterative method for computing the stabilizing solution of game-theoretic algebraic Riccati equations in a deterministic framework. \cite{Feng2010,Dragan2011,Dragan2017,Ivanov2018} have continuously made progress in the numerical calculation of stabilizing solutions for a broader class of SGTRDEs. \cite{Dragan2015} further extended it to the  cases of the SGTRDEs associated with stochastic $H_\infty$ problems and the controlled systems described by It\^o differential equations with periodic coefficients. 

To date, no algorithm has been reported in the existing literature for ZSLQSDG-type SGTRDEs with periodic coefficients. This paper develops a new algorithm via the construction of dual-layer matrix iteration sequences, which reformulate the original problem into a set of interconnected bilevel subproblems. These subproblems are solved sequentially by identifying the maximal periodic solutions to their associated Riccati differential equations, ultimately yielding stabilizing periodic solutions for the original problem. 
This idea originates from defect correction method proposed in \cite{Mehrmann1988}. Importantly, this study extends classical settings and develops a unified algorithm applicable to three classes of problems with periodic coefficients: deterministic Game-Theoretic Riccati Differential Equations, stochastic $H_\infty$-type SGTRDEs, and the more general ZSLQSDG-type SGTRDEs.

Remaining content outline: Section \ref{sec:Preliminaries} covers ZSLQSDG's mathematical framework, algorithm iteration process, and a class of auxiliary problems. Section \ref{sec:main_results} presents main results: inner-outer matrix iteration sequence construction logic, proofs of existence, uniqueness and boundedness of stabilizing periodic solutions, and algorithm convergence. Section \ref{sec:numerical_experiments} verifies algorithm effectiveness and stability via numerical experiments.

\section{Preliminaries} 
\label{sec:Preliminaries}

\subsection{Notation}
Let us first introduce the following notation used in this paper:
\begin{itemize}
    \item $\mathbb{R}_+:[0,\infty)$; $\mathbb{R}^n$: $n$-dimensional Euclidean space; $\mathbb{R}^{n\times m}$: the set of all $n\times m$ real matrices; $\mathbb{S}^n$: the set of all $n\times n$ symmetric matrices; $\overline{\mathbb{S}}_{+}^n$: the set of all $n\times n$ symmetric positive semi-definite matrices; $\mathbb{S}_{+}^n$: the set of all $n\times n$ symmetric positive definite matrices; $\mathbf{C}^1(\mathbb{R}_+,\mathbb{S}^n)$: the set of all continuously differentiable functions of order 1 defined on the interval $\mathbb{R}_+$ with values in $ \mathbb{S}^n$; $\mathbf{B}(\mathbb{S}^n)$: the space of linear operators defined on $\mathbb{S}^n$.
    \item $I_n$: the identity matrix of size $n$; $\mathbb{O}_{n \times m}$: the null matrix of size $n \times m$. It can be simplified as 0 when no ambiguity is generated; $\mathbb{O}_{m \times n}^{\times (l)}$ denotes the object represented by $\mathbb{O}_{m \times n}$ arranged repeatedly for $l$ times.
    \item $A^{\top}$: the transpose of the matrix $A$; $\mathrm{Tr}(\cdot)$: the trace of a square matrix; $\langle\cdot,\cdot\rangle$: the inner product on a Hilbert space. In particular, the usual inner product on $\mathbb{R}^{n\times m}$ is given by $\langle A,B\rangle\mapsto\mathrm{Tr}(A^{\top}B)$.
    \item If $A\in\mathbb{S}_{+}^n$ (resp., $A\in\overline{\mathbb{S}}_{+}^n$) is a symmetric positive definite (resp., symmetric positive semi-definite) matrix, we write $A\succ0$ (resp., $A\succeq0$). For any $A,B\in\mathbb{S}^n$, we use the notation $A\succ B$ (resp., $A\succeq B$) to indicate that $A - B\succ0$ (resp., $A - B\succeq0$).
\end{itemize}

\subsection{A Class of Stochastic Game-Theoretic Riccati Differential Equations arising in Zero-Sum Linear-Quadratic Stochastic Differential Games}
We are interested in a class of Stochastic Game-Theoretic Riccati Differential Equations arising from the following problem. Let $(\Omega,\mathcal{F},\mathbb{F},\mathbb{P})$ be a complete filtered probability space on which a r-dimensional standard Brownian motion $W = \{(w_{1}(t),\dots,w_{r}(t))^{\top}; t\geq0\}$ is defined with $\mathbb{F}=\{\mathcal{F}_t\}_{t\geq0}$ being the usual augmentation of the natural filtration generated by $W$. We consider the following controlled linear stochastic differential equation (SDE) on $\mathbb{R}_+$:
\begin{equation} 
\label{lqzsdg:sde}
    \begin{cases} 
        dx(t) = \lbrack A_0(t)x(t) + B_{01}(t) u_1(t) + B_{02}(t) u_2(t) \rbrack dt + \sum_{k=1}^{r}\lbrack A_{k}(t)x(t) + B_{k1}(t) u_1(t) + B_{k2}(t)u_2(t) \rbrack dw_{k}(t) \\
        x(0) = x_0
    \end{cases}
\end{equation}
in which $x_0\in\mathbb{R}^n$, $u(t) = \begin{pmatrix} u_1(t) \\ u_2(t) \end{pmatrix}$ and $B_k(t, i) = \begin{pmatrix} B_{k1}(t) & B_{k2}(t) \end{pmatrix};B_{ki}(t) \in \mathbb{R}^{n \times m_i}(i = 1, 2;k=0,\dots,r)$. In the above, the state process $x$ is an $n$-dimensional vector-valued function, player 1 $u_1$ and player 2 $u_2$ are an $m_1$-dimensional vector and an $m_2$-dimensional vector-valued functions, respectively. 

Player 1 and Player 2 share the same performance functional:
\begin{equation}
\label{pf:lqzsdg_pf}
    J(x_0;u_1,u_2)\triangleq\mathbb{E}\int_{0}^{\infty}\Bigg[\Bigg\langle\begin{pmatrix}
    M(t) & L_1(t) & L_2(t)\\
    L_1^{\top}(t) & R_{11}(t) & R_{12}(t)\\
    L_2^{\top}(t) & R_{12}^{\top}(t) & R_{22}(t)
    \end{pmatrix}\begin{pmatrix}
    x(t)\\
    u_1(t)\\
    u_2(t)
    \end{pmatrix},\begin{pmatrix}
    x(t)\\
    u_1(t)\\
    u_2(t)
    \end{pmatrix}\Bigg\rangle
    \Bigg]dt
\end{equation}
where $L(t) = \begin{pmatrix} L_1(t) & L_2(t) \end{pmatrix};\, L_i(t) \in \mathbb{R}^{n \times m_i}(i = 1, 2;m_1+m_2=m)$, $R(t) = \begin{pmatrix} R_{11}(t) & R_{12}(t) \\ R_{12}^{\top}(t) & R_{22}(t) \end{pmatrix};\, R_{ij}(t) \\\in \mathbb{R}^{m_i \times m_j}(i, j = 1, 2)$.

In this zero-sum game, Player 1 (\textit{the maximizer}) selects control $u_1$ to maximize \eqref{pf:lqzsdg_pf}, while Player 2 (\textit{the minimizer}) chooses $u_2$ to minimize the same function. The problem is to find an admissible control pair $(u_1^*,u_2^*)$ that both players can accept. For a description of ZSLQSDG, refer to \cite{Dragan2020} and \cite{Sun2020book}. More detailed information can be found therein. 

\begin{remark}
    \Cite{Dragan2020} discusses some properties of the solution set satisfying the sign condition \ref{pf:sign_conditions_1} under certain natural assumptions, and points out its correspondence with the classical ZSLQSDG problem in Section 3. The parameter setting refers to Assumptions 1 in the paper. 
\end{remark}

\begin{remark}
Employing Lemma 2 in Chapter 4 from the work of \cite{halanay2012book}, we obtain the following:
If $ X(\cdot)$ is be a $\theta$-periodic solution of the SGTRDE \eqref{pf:gtrde} and $X(t) \in \mathbb{S}^n_+$ for all $t \in \mathbb{R}_{+}$, then the following are equivalent:
\begin{enumerate}
    \item[(i)] the solution $X(\cdot)$ satisfies the sign condition \ref{pf:sign_conditions_1};
    \item[(ii)] the solution $X(\cdot)$ satisfies the sign condition: \footnote{To simplify the notation for subsequent content, whenever $(t,X(t))$ is involved, it is abbreviated as $(t,X_t)$ by default. For example, $\mathbb{R}_{22}(t,X_t)=\mathbb{R}_{22}(t,X(t))$.}
    \begin{equation}
    \label{pf:sign_conditions_1e}
    \begin{aligned}
    &\mathbb{R}_{22}(t, X_t) = R_{22}(t) + \sum_{k=1}^{r} B^{\top}_{k2}(t) X(t) B_{k2}(t) \succ 0;\forall t \in \mathbb{R}_{+}, \\
    &\mathbb{R}_{22}^{\sharp}(t, X_t) = R_{11}(t) + \sum_{k=1}^{r} B_{k1}^{\top}(t) X(t) B_{k1}(t) - \left[ R_{12}(t) + \sum_{k=1}^{r} B_{k1}^{\top}(t) X(t) B_{k2}(t) \right] \mathbb{R}_{22}(t, X_t)^{-1} \\
    &\quad \left[ R_{12}(t) + \sum_{k=1}^{r} B_{k1}^{\top}(t) X(t) B_{k2}(t) \right]^{\top} \prec 0;\forall t \in \mathbb{R}_{+}.
    \end{aligned}
    \end{equation}
\end{enumerate}
\end{remark}

\begin{definition}
A global solution $\tilde{X}(\cdot): \mathbb{R}_+ \to \mathbb{S}^n$ of the SGTRDE \eqref{pf:gtrde} is called a \textit{stabilizing solution} if the following conditions hold:
\begin{enumerate}
    \item[(i)] $\inf_{t \in \mathbb{R}_+} \left| \det\left( R(t) + \sum_{k=1}^r B_k^{\top}(t)\tilde{X}(t)B_k(t) \right) \right| > 0$;  
    (i.e., the matrix $R(t) + \sum_{k=1}^r B_k^{\top}(t)\tilde{X}(t)B_k(t)$ is uniformly invertible on $\mathbb{R}_+$.)
    \item[(ii)] The closed-loop system
    \begin{equation*}
    dx(t) = \left[ A_0(t) + B_0(t)F(t,\tilde{X}_t) \right]x(t)dt + \sum_{k=1}^r \left[ A_k(t) + B_k(t)F(t,\tilde{X}_t) \right]x(t)dw_k(t)
    \end{equation*}
    is such that its zero solution is exponentially stable in mean square. This is denoted concisely as the system $\left( A_0(\cdot) + B_0(\cdot)F(\cdot,\cdot), \, A_1(\cdot) + B_1(\cdot)F(\cdot,\cdot), \, \dots, \, A_r(\cdot) + B_r(\cdot)F(\cdot,\cdot) \right)$ being stable, where $F(t,X_t)$ for all $t \in \mathbb{R}_+$ is defined by
    \begin{equation}
    \label{pf:esms_f}
    F(t,X_t) = - \left[ R(t) + \sum_{k=1}^r B_k^{\top}(t)X(t)B_k(t) \right]^{-1} \left[ B_0^{\top}(t)X(t) + \sum_{k=1}^r B_k^{\top}(t)X(t)A_k(t) + L^{\top}(t) \right].
    \end{equation}
\end{enumerate}
\end{definition}

\subsection{Algorithm Design}
\label{sec:algorithm_design}

In this paper, we aim to compute the the stabilizing solution numerically for SGTRDEs arising in classical ZSLQSDG with periodic time-varying parameter.
Consider the sequences $\{X^{(h)}(\cdot)\}_{h\geq0}$ and $\{Z^{(h)}(\cdot)\}_{h\geq0}$ constructed according to the following procedure:  
\begin{enumerate}
    \item [1.]
    For $h = 0$, set $X^{(0)}(\cdot) = 0$, and let $Z^{(0)}(\cdot)$ be the unique $\theta$-periodic and stabilizing solution to the following Riccati differential equations with the definite-sign quadratic term:
    \begin{equation}
    \label{algorithm:initial}
        \begin{aligned}
        \frac{d}{dt}Z^{(0)}(t)&+ A^{\top}_{0,(0)}(t) Z^{(0)}(t)+Z^{(0)}(t)A_{0,(0)}(t)+\sum_{k=1}^{r}A^{\top}_{k,(0)}(t) Z^{(0)}(t)A_{k,(0)}(t) + M(t) - L(t)R(t)^{-1}L^{\top}(t)\\
        &- \left[ Z^{(0)}(t)B_{02}(t)+ \sum_{k=1}^{r} A^{\top}_{k,(0)}(t)Z^{(0)}(t)B_{k2}(t) \right] \left[R_{22}(t) + \sum_{k=1}^{r}B_{k2}^{\top}(t) Z^{(0)}(t) B_{k2}(t)\right]^{-1} \\
        &\quad \times \left[ B_{02}^{\top}(t)Z^{(0)}(t) + \sum_{k=1}^{r} B_{k2}^{\top}(t)Z^{(0)}(t)A_{k,(0)}(t) \right]= 0, \quad t \geq 0.     
        \end{aligned}
    \end{equation}
    where $A_{k,(0)}(t) = A_{k}(t) + B_{k}(t)F(t,0)(\forall t \in \mathbb{R}_+,k=0,\dots,r).$
    \item [2.]
    For $h \geq 1$, update $X^{(h)}(\cdot)$ via: 
    \begin{equation}
    \label{algorithm:external_circulation}
            \begin{aligned}
                X^{(h)}(t) = X^{(h-1)}(t) + Z^{(h-1)}(t), \forall t \in \mathbb{R}_+,     
            \end{aligned} 
    \end{equation}
    and solving $Z^{(h)}(\cdot)$, which is the unique $\theta$-periodic and stabilizing solution to the following Riccati differential equations with the definite-sign quadratic term:
    \begin{equation}
    \label{algorithm:internal_circulation}
        \begin{aligned}
        &\frac{d}{dt}(P^{(h)}(t)+Z^{(h)}(t))+ A^{\top}_{0,(h)}(t) Z^{(h)}(t)+Z^{(h)}(t)A_{0,(h)}(t)+\sum_{k=1}^{r}A^{\top}_{k,(h)}(t) Z^{(h)}(t)A_{k,(h)}(t) \\
        &+ V^{\top}_{(h)}(t)\mathbb{R}_{22}^{\sharp}(t, X^{(h)}_t)^{-1}V_{(h)}(t)- \left[ Z^{(h)}(t)B_{02}(t) + \sum_{k=1}^{r} A^{\top}_{k,(h)}(t)Z^{(h)}(t)B_{k2}(t) \right]\\
        &\quad \times \left[R_{22}(t) + \sum_{k=1}^{r}B_{k2}^{\top}(t) Z^{(h)}(t) B_{k2}(t)\right]^{-1} \left[ B_{02}^{\top}(t)Z^{(h)}(t) + \sum_{k=1}^{r} B_{k2}^{\top}(t)Z^{(h)}(t)A_{k,(h)}(t) \right]= 0, \,t \geq 0.     
        \end{aligned}
    \end{equation}
    where $A_{k,(h)}(t) = A(t) + B(t)F(t,X^{(h)}_t)(k=0,\dots,r)$ and $V_{(h)}(t)=$
    \begin{equation}
    \label{algorithm:vxz}
    \begin{pmatrix}I_{m_1} & -R_{21}(t,X^{(h)}_t)R_{22}(t,X^{(h)}_t)^{-1}\end{pmatrix}[B^{\top} _{0}(t)Z^{(h-1)}(t) + \sum_{k=1}^{r}B^{\top}_{k}(t)Z^{(h-1)}(t)A_{k,(h-1)}(t)],\forall t \in \mathbb{R}_+. 
    \end{equation}
\end{enumerate}

In the subsequent content, we will show:
\begin{itemize}
    \item The sequences $\{Z^{(h)}(\cdot)\}$ and $\{X^{(h)}(\cdot)\}$ are well-defined, and for each $h \geq 0$, the corresponding Riccati differential equations admits a unique $\theta$-periodic, stabilizing solution $Z^{(h)}(\cdot)$ with $Z^{(h)}(t) \in \overline{\mathbb{S}}_{+}^n$ for all $t \in \mathbb{R}_+$.
    \item These sequences are convergent and we have  
    \[
        \lim_{h\to\infty} X^{(h)}(t) = \tilde{X}(t), \, \lim_{k\to\infty} Z^{(h)}(t) = 0,\,\forall t \in \mathbb{R}_+,
    \] 
    where $\tilde{X}(\cdot)$ is unique $\theta$-periodic and stabilizing solution to the SGTRDE \eqref{pf:gtrde}.
\end{itemize}

\subsection{A class of Auxiliary Problems}

In order to facilitate the analysis of SGTRDEs \eqref{pf:gtrde}, we introduce a fundamental transformation of the control variable. Specifically, we first eliminate the bias term in the problem through transformation, then parameterize Player 2's control as a feedback law that depends linearly on both the state $x(t)$ and the opponent's control. This approach enables us to embed the original problems into the auxiliary problems, decoupling the problem structure without altering their essence.

Setting formally 
$
\begin{bmatrix}  v_1(t) \\ v_2(t) \end{bmatrix}=\begin{bmatrix}  u_1(t) \\ u_2(t) \end{bmatrix}+ R(t)^{-1} L^{\top}(t)
$ 
and $v_2(t) = K(t)x(t) + W(t)v_1(t)$ for all $t \in \mathbb{R}_+$, we obtain
\begin{equation} 
\label{pr:sde_kw}
    \begin{cases} 
        dx(t) = \lbrack A_{0K}(t)x(t) + B_{0W}(t)v_1(t) \rbrack dt + \sum_{k=1}^{r}\lbrack A_{kK}(t)x(t) + B_{kW}(t)v_1(t) \rbrack dw_{k}(t) \\
        x(0) = x_0
    \end{cases}
\end{equation}
in which $x_0\in\mathbb{R}^n$, and 
\begin{equation}
\label{pr:pf_kw}
    J_{KW}(x_0;u_1)\triangleq\mathbb{E}\int_{0}^{\infty}\Bigg[\Bigg\langle\begin{pmatrix}
    M_K(t) & L_{KW}(t) \\
    L_{KW}^{\top}(t) & R_W(t)
    \end{pmatrix}\begin{pmatrix}
    x(t)\\
    v_1(t)
    \end{pmatrix},\begin{pmatrix}
    x(t)\\
    v_1(t)
    \end{pmatrix}\Bigg\rangle
    \Bigg]dt,
\end{equation}
where
\[
\begin{cases}
A_{kK}(t) = A_k(t) + B_{k}(t)F(t,0) + B_{k2}(t)K(t), & k=0,\dots,r \\
B_{kW}(t)= B_{k1}(t) + B_{k2}(t)W(t), & k=0,\dots,r \\
M_K(t) = M(t) + K^{\top}(t)R_{12}(t)K(t) \\
L_{KW}(t) = K^{\top}(t)R_{12}^{\top}(t) + K^{\top}(t)R_{22}(t)W(t) \\
R_W(t) = \begin{pmatrix} I_{m_1} \\ W(t) \end{pmatrix}^{\top} \begin{pmatrix} R_{11}(t) & R_{12}(t) \\ R_{12}^{\top}(t) & R_{22}(t) \end{pmatrix} \begin{pmatrix} I_{m_1} \\ W(t) \end{pmatrix}
\end{cases}
\]
The corresponding Riccati differential equations is
\begin{equation}
\label{pr:sgrde_kw}
\begin{aligned}
\frac{d}{dt}{Y}(t) &+ A_{0K}^{\top}(t)Y(t) + Y(t)A_{0K}(t) + \sum_{k=1}^r A_{kK}^{\top}(t)Y(t)A_{kK}(t)+ M_K(t) \\
&- \left[Y(t)B_{0W}(t) + \sum_{k=1}^r A_{kK}^{\top}(t)Y(t)B_{kW}(t) + L_{KW}(t)\right] \left[R_W(t) + \sum_{k=1}^r B_{kW}^{\top}(t)Y(t)B_{kW}(t)\right]^{-1}\\
&\quad \times \left[ B^{\top}_{0W}(t)Y(t) + \sum_{k=1}^r B^{\top}_{kW}(t)Y(t)A_{kK}^{\top}(t) + L^{\top}_{KW}(t) \right] = 0,\quad t \geq 0.
\end{aligned} 
\end{equation}
Throughout this work, parameters associated with the SGTRDEs in \eqref{pf:gtrde} are formally denoted by $\Sigma$. Let $\mathcal{A}^{\Sigma}$ stands for the set of all continuous and $\theta$-periodic functions $(K(\cdot), W(\cdot))$, where $t \to K(t) : \mathbb{R}_+ \to \mathbb{R}^{m_2 \times n} $ and $t \to W(t) : \mathbb{R}_+ \to \mathbb{R}^{m_2 \times m_1} $ satisfy: 
\begin{itemize}
    \item The system $(A_{0K},\dotsb,A_{rK})$ is stable;
    \item The corresponding Riccati differential equations \eqref{pr:sgrde_kw} has a $\theta$-periodic and stabilizing solution $\tilde{Y}_{KW}(\cdot)$, satisfying the sign conditions
    \[
    R_W(t) + \sum_{k=1}^r B_{kW}^{\top}(t)\tilde{Y}_{KW}(t)B_{kW}(t) \prec 0, \forall t \in \mathbb{R}_+ .
    \]
\end{itemize}

\section{The Main Results}
\label{sec:main_results}

\subsection{Analysis of Structural Characteristics for SGTRDEs}

From the SGTRDEs \eqref{pf:gtrde}, we can define a mapping $\mathcal{G}: \mathrm{Dom}\,\mathcal{G} \to \mathbb{S}^n$ that satisfies the following relation:
\begin{equation}
\label{main_results:function_g}
\begin{aligned}
    & \mathcal{G}(t,X) =A_0^{\top}(t)X + XA_0(t) + \sum_{k=1}^r A_k^{\top}(t)XA_k(t)+ M(t) \\
    &- \left[ XB_0(t) + \sum_{k=1}^r A_k^{\top}(t)XB_k(t) + L(t) \right] \left[ R(t) + \sum_{k=1}^r B_k^{\top}(t)XB_k(t) \right]^{-1} \left[ B^{\top}_0(t)X + \sum_{k=1}^r B^{\top}_k(t)XA_k(t) + L^{\top}(t) \right].     
\end{aligned}
\end{equation}
The nonlinear function $\mathcal{G}$ is well-defined on the subset:
\[
    \mathrm{Dom}\,\mathcal{G}=\left\{(t,X) \in \mathbb{R}_+ \times \mathbb{S}^n \,| \, \mathbb{R}_{22}^{\sharp}(t, X) \prec 0 \,\, \text{and} \,\, \mathbb{R}_{22}(t, X) \succ 0 \right\}.   
\]
Now we rewrite the SGTRDEs \eqref{pf:gtrde} in the following compact form:
\begin{equation*}
\begin{aligned}
    \frac{d}{dt}{X}(t)+\mathcal{G}(t,X_t) =0, \,\, t \geq 0.     
\end{aligned}
\end{equation*}
The mapping $\mathcal{G}$ serves as the fundamental basis for algorithm construction. First, we conduct a analysis on SGTRDEs and present its relevant structural properties of $\mathcal{G}$. Building on these properties, we transform the original problem into related subproblems and establish a sequence of interrelated subproblem Riccati differential equations.

\begin{proposition}
\label{function_g:proposition_1}
Let $(t,X),(t,X+Z) \in \mathrm{Dom}\,\mathcal{G}$, then
\begin{align*}
    F(t,X+Z)=F(t,X)- R(t,X+Z)^{-1} N(t,X,Z),
\end{align*}
where 
\begin{align*}
R(t,X) =R(t) + \sum_{k=1}^r B_k^{\top}(t)XB_k(t),\,
N(t,X,Z)=B_{0}^{\top}(t)Z + \sum_{k=1}^{r}B_{k}^{\top}(t)Z(A_k(t)+ B_{k}(t)F(t,X)).
\end{align*}
\end{proposition}

\begin{proof}
Define the increment $\Delta F (t,X,Z) = F(t, X+Z) - F(t, X)$.
By the definition of $F(\cdot,\cdot)$ in \eqref{pf:esms_f} and multiplying both sides by $R(t, X+Z) $ we obtain: 
\[
R(t, X+Z) \Delta F (t,X,Z) = R(t, X+Z) F(t, X+Z) - R(t, X+Z) F(t, X).
\]
Note that
\[
R(t, X+Z) = R(t, X) + \Delta R (t, Z), \,\, \text{where} \,\, \Delta R (t, Z)= \begin{bmatrix}
\sum_{k=1}^{r} B_{k1}^{\top}(t) Z B_{k}(t) \\
\sum_{k=1}^{r} B_{k2}^{\top}(t) Z B_{k}(t) 
\end{bmatrix}.
\]
Thus, $ R(t, X+Z) \Delta F (t,X,Z)  = -S(t, X+Z) + S(t, X) - \Delta R(t, Z) F(t, X)$
\begin{align*}
=&- \begin{bmatrix} B_{01}^{\top}(t) Z + \sum_{k=1}^{r} B_{k1}^{\top}(t) Z A_k(t) \\ B_{02}^{\top}(t) Z + \sum_{k=1}^{r} B_{k2}^{\top}(t) Z A_k(t) \end{bmatrix}-\begin{bmatrix} \sum_{k=1}^{r} B_{k1}^{\top}(t) Z B_k(t) F(t, X) \\ \sum_{k=1}^{r} B_{k2}^{\top}(t) Z B_k(t) F(t, X) \end{bmatrix}\\
=& - \begin{bmatrix} B_{01}^{\top}(t) Z + \sum_{k=1}^{r} B_{k1}^{\top}(t) Z \left( A_k(t) + B_k(t) F(t, X) \right) \\ B_{02}^{\top}(t) Z + \sum_{k=1}^{r} B_{k2}^{\top}(t) Z \left( A_k(t) + B_k(t) F(t, X) \right) \end{bmatrix},   
\end{align*}
where $S(t, X)=B_{0}^{\top}(t)X + \sum_{k=1}^{r}B_{k}^{\top}(t)XA_k(t)+ L^{\top}(t)$. 

Multiplying both sides by $R(t, X+Z)^{-1}$ yields:
\[
F(t, X+Z) = F(t, X) - R(t, X+Z)^{-1} \begin{bmatrix} B_{01}^{\top}(t) Z + \sum_{k=1}^{r} B_{k1}^{\top}(t) Z \left( A_k(t) + B_k(t) F(t, X) \right) \\ B_{02}^{\top}(t) Z + \sum_{k=1}^{r} B_{k2}^{\top}(t) Z \left( A_k(t) + B_k(t) F(t, X) \right) \end{bmatrix}.
\]
\end{proof}

\begin{proposition}[\cite{Dragan2013book}]
\label{function_g:proposition_2}    
Let $(t,X) \in \mathrm{Dom}\,\mathcal{G}$, then
\begin{equation*}
    \begin{aligned}
        \mathcal{G}&(t,X)=(A_0(t) + B_{0}(t)\varTheta)^{\top}X+X(A_0(t) + B_{0}(t)\varTheta)\\
        &+ \sum_{k=1}^{r}(A_k(t) + B_{k}(t)\varTheta)^{\top} X (A_k(t) + B_{k}(t)\varTheta)+ M(t)+ \varTheta^{\top}R(t)\varTheta+ \varTheta^{\top}L^{\top}(t) + L(t)\varTheta\\
        &- (F(t,X)-\varTheta)^{\top} R(t,X) (F(t,X)-\varTheta),
    \end{aligned}
\end{equation*}
for any $\varTheta \in \mathbb{R}^{m \times n}$, where $F(t,X)$ is defined in \eqref{pf:esms_f}. 
\end{proposition}

\begin{proposition}
\label{function_g:proposition_3}
Let $(t,X),(t,X+Z) \in \mathrm{Dom}\,\mathcal{G}$, then 
\begin{equation}
\label{function_g:proposition_1_xz}   
\begin{aligned}
    &\mathcal{G}(t,X+Z) = \mathcal{G}(t,X)+ \left( A_0(t)+B_{0}(t)F(t,X) \right)^{\top} Z + Z \left( A_0(t)+ B_{0}(t)F(t,X) \right)\\
    &+\sum_{k=1}^{r}\left( A_k(t)+ B_{k}(t)F(t,X) \right)^{\top} Z \left(A_k(t)+ B_{k}(t)F(t,X) \right)- N^{\top}(t,X,Z) R(t,X+Z)^{-1} N(t,X,Z),
\end{aligned}    
\end{equation}
where $F(t,X)$ is defined in \eqref{pf:esms_f} and
\begin{align*}
    R(t,X)= R(t) + \sum_{k=1}^r B_k^{\top}(t)X B_k(t),
    N(t,X,Z)=B_{0}(t)Z + \sum_{k=1}^{r}B_{k}^{\top}(t)Z(A_k(t)+ B_{k}(t)F(t,X)).
\end{align*}
\end{proposition}

\begin{proof}
From the definition of $\mathcal{G}$, we have:
\begin{align*}
    \mathcal{G}(t, X + Z) =&  A^{\top}_0(t) (X+Z) + (X+Z) A_0(t) + \sum_{k=1}^{r} A^{\top}_k(t) (X+Z) A_k(t) + M(t) \\
    &- S^{\top}(t, X + Z) R(t, X + Z)^{-1} S(t, X + Z).
\end{align*}
where $S(t, X)=B_{0}^{\top}(t)X + \sum_{k=1}^{r}B_{k}^{\top}(t)XA_k(t)+ L^{\top}(t)$ and $R(t, X )=R(t) + \sum_{k=1}^{r}B_{k}^{\top}(t)XB_{k}(t)$.
From Proposition \ref{function_g:proposition_2}, expanding and rearranging terms, we obtain:
\begin{align*}
   \mathcal{G}(t, X + Z)=& A_0(t)^{\top} X + X A_0(t) + \sum_{k=1}^{r} A_k(t)^{\top} X A_k(t) + M(t) + F^{\top}(t, X )R(t, X)F(t, X )\\
   &F^{\top}(t, X )S(t, X)+S^{\top}(t, X)F(t, X )+(A_0(t)+B_0F(t, X))^{\top} Z\\
   &+ Z (A_0(t)+B_0F(t, X )) + \sum_{k=1}^{r} (A_k(t)+B_kF(t, X ))^{\top} Z (A_k(t)+B_kF(t, X ))\\
   &- \left[ F(t, X + Z)- F(t, X) \right]^{\top} R(t, X + Z) \left[ F(t, X + Z)- F(t, X) \right]
\end{align*}
From Proposition \eqref{function_g:proposition_1}, we derive:
\begin{align*}
\left[ F(t, X + Z)- F(t, X) \right]^{\top} R(t, X + Z) \left[ F(t, X + Z)- F(t, X) \right]=N^{\top}(t,X,Z) R(t,X+Z)^{-1} N(t,X,Z).
\end{align*}
Thus,
\begin{align*}
&\mathcal{G}(t,X+Z) = \mathcal{G}(t,X)+ \left( A_0(t)+B_{0}(t)F(t,X) \right)^{\top}Z+Z \left( A_0(t)+ B_{0}(t)F(t,X) \right)\\
&+\sum_{k=1}^{r}\left( A_k(t)+ B_{k}(t)F(t,X) \right)^{\top}Z\left(A_k(t)+ B_{k}(t)F(t,X) \right)- N(t,X,Z) R(t,X+Z)^{-1} N^{\top}(t,X,Z).
\end{align*}
\end{proof}

\begin{lemma}
\label{pr:lemma}    
Assume that $R_{22}(t) \succ 0$ for all $t \in \mathbb{R}_+$ and $\mathcal{A}^{\Sigma}$ is not empty. Then we have  
\begin{itemize}
    \item[(i.)] $(t, 0) \in \text{Dom}\,\mathcal{G}$ for all $t \in \mathbb{R}_+$. 
    \item[(ii.)] $(t, \tilde{Y}_{KW}(t)) \in \text{Dom}\,\mathcal{G}$ for all $t \in \mathbb{R}_+$, where $\tilde{Y}_{KW}(\cdot)$ be the $\theta$-periodic and stabilizing solution to the corresponding Riccati differential equation \eqref{pr:sgrde_kw}.
\end{itemize}
\end{lemma}

\begin{proof}
(i.) The conclusion is directly derived from Part (b) of Remark 3 of \cite{Dragan2020}.

(ii.) Since $R_{22}(t) \succ 0$ and $\tilde{Y}_{KW}(t) \succeq 0$ (due to the stabilizing solution property) for all $t \in \mathbb{R}_+$, we obtain $\mathbb{R}_{22}(t,\tilde{Y}_{KW}(t) ) \succ 0$ for all $t \in \mathbb{R}_+$.
Define $R_{ij}(t,\tilde{Y}_{KW}(t))=R_{ij}(t)+ \sum_{k=1}^{r} D^{\top}_{ki}(t)\tilde{Y}_{KW}(t)D_{kj}(t)(i,j=1,2)$ and
\[
\begin{bmatrix} U_{11}(t) & U_{12}(t)  \\ U_{21}(t)  & U_{22}(t)  \end{bmatrix} = \begin{bmatrix} I_{m_1} & W^{\top}(t)  \\ 0 & I_{m_2} \end{bmatrix} \begin{bmatrix} R_{11}(t,\tilde{Y}_{KW}(t)) & R_{12}(t,\tilde{Y}_{KW}(t)) \\ R_{21}(t,\tilde{Y}_{KW}(t)) & R_{22}(t,\tilde{Y}_{KW}(t)) \end{bmatrix}  \begin{bmatrix} I_{m_1} & 0 \\ W(t) & I_{m_2} \end{bmatrix}.
\]
Then
\[
U_{22}(t) \succ 0, U_{11}(t)  - U_{12} ^{\top}(t)  U_{22}(t) ^{-1}U_{21}(t)  = R_W(t)  + \sum_{l=1}^r D_{lW}^{\top}(t) \tilde{Y}_{KW}(t) D_{lW}(t)  \prec 0, \forall \, t \in \mathbb{R}_+ .
\]
By Lemma 6.2 in \cite{Dragan1994}, this implies
\[
R_{11}(t,\tilde{Y}_{KW}(t))-R^{\top}_{12}(t,\tilde{Y}_{KW}(t))R_{22}(t,\tilde{Y}_{KW}(t))^{-1}R_{21}(t,\tilde{Y}_{KW}(t)) \prec 0, \forall \, t \in \mathbb{R}_+ .
\]
Thus $(t, \tilde{Y}_{KW}(t)) \in \text{Dom}\,\mathcal{G} $ for all $t \in \mathbb{R}_+$.
\end{proof}

\subsection{Discriminant condition for stabilizing solutions of Riccati differential equations with definite-sign quadratic term}

From the structure of $\mathcal{G}$ and Proposition \ref{function_g:proposition_3}, a sequence of interrelated sub Riccati differential equations with definite-sign quadratic term can be constructed, as the algorithm described in the \secref{sec:algorithm_design}. A key requirement here is the existence and uniqueness of stabilizing solutions to these differential equations, for which we present a criterion in what follows.

\begin{definition}[\cite{Dragan2013book}]
The system \eqref{lqzsdg:sde} denoted as 
\[
\left[ A_0(\cdot),\dotsb,A_r(\cdot); B_{0}(\cdot), \dotsb,  B_{r}(\cdot) \right],
\]
is called \textit{stochastically stabilizable} if there exists a continuous function $\varTheta(\cdot) : \mathbb{R}_+ \to \mathbb{R}^{m \times n}$ such that the system $\left( A_{0}(\cdot)+B_{0}(\cdot)\varTheta(\cdot),\dotsb, A_{r}(\cdot)+B_{r}(\cdot)\varTheta(\cdot) \right)$ is stable.
\end{definition}

\begin{definition}
Consider the following stochastic observation system:
\begin{equation}
\label{stochastic_observation_system}
    \begin{cases}
         dx(t) = A_0(t)x(t)dt + \sum_{k=1}^{r}A_k(t)x(t)dw_k(t) \\
         dy(t) = C_0(t)x(t)dt + \sum_{k=1}^{r}C_k(t)x(t)dw_k(t)
    \end{cases}
\end{equation}
where $t \mapsto C_k(t) : \mathbb{R}_+ \to \mathbb{R}^{q \times n}(k=0,\dots,r)$.
We denote this system by $\left[ C_{0}(\cdot), \dots, C_{r}(\cdot); A_{0}(\cdot), \dots, A_{r}(\cdot) \right]$. 
If there exists continuous function $\varTheta(\cdot) : \mathbb{R}_+ \to \mathbb{R}^{n \times p}$ such that the system $(A_{0}(\cdot)+\varTheta(\cdot) C_{0}(\cdot),\dotsb, A_{r}(\cdot)+\varTheta(\cdot)C_{r}(\cdot))$ is stable, the system $\left[ C_{0}(\cdot), \dotsb, C_{r}(\cdot); A_{0}(\cdot), \dotsb, A_{r}(\cdot) \right]$ is called \textit{stochastically detectable}.
\end{definition}

\begin{remark}
For more details regarding stochastic detectability for systems described by It\^o differential equations, we refer to Chapter 4 in \cite{Dragan2013book}. 
\end{remark}

\begin{lemma}
\label{main_results:stochastically_detectable_lemma}
If the system $[C_{0}(\cdot), \dotsb, C_{r}(\cdot); A_{0}(\cdot), \dotsb, A_{r}(\cdot)]$ is stochastically detectable, the following are equivalent:
    \begin{itemize}
        \item[a.]  Then the system $(A_{0}(\cdot), \dotsb, A_{r}(\cdot))$ is stable.
        \item[b.]  $\frac{d}{dt}X(t) + A_0(t)^{\top}X(t)+ X(t)A_0(t) + \sum_{k=1}^r A_k(t)^{\top} X(t)A_k(t)+\sum_{k=0}^{r}C_{k}^{\top}(t)C_{k}(t)=0,t \ge 0$ has a solution $X(t)$ with $X(t) \in \overline{\mathbb{S}}_{+}^n$ for all $t \in \mathbb{R}_+$.
    \end{itemize}
\end{lemma}

\begin{proof}
$a. \implies b.$
Since $\sum_{k=0}^{r}C_{k}^{\top}(t)C_{k}(t) \succeq 0$ for all $t \in \mathbb{R}_+$ and the system $(A_{0}(\cdot), \dotsb, A_{r}(\cdot))$ is stable, it follows from Theorem 2.7.5. in \cite{Dragan2013book} that the differential equation
$\frac{d}{dt}X(t) + A_0(t)^{\top}X(t)+ X(t)A_0(t) + \sum_{k=1}^r A_k(t)^{\top} X(t)A_k(t)+\sum_{k=0}^{r}C_{k}^{\top}(t)C_{k}(t)=0,t \ge 0$
admits a bounded solution $X(t) \in \overline{\mathbb{S}}_{+}^n$ for all $t \in \mathbb{R}_+$.

$b. \implies a.$ is a direct application of Remark 4.1.5 in \cite{Dragan2013book}. For a detailed proof, refer to Theorem 4.1.7 and Remark 4.1.5 in the aforementioned book.
\end{proof}

Define $\mathbf{C}_b^1(\mathbb{R}_+,\mathbb{S}^n) = \left\{ X(\cdot) \in \mathbf{C}^1(\mathbb{R}_+, \mathbb{S}^n) \mid X(\cdot), \frac{d}{dt}X(\cdot) \,\text{are bounded} \right\}$ and the operator $\boldsymbol{\Lambda}(t,X): \mathbb{R}_+ \times \mathbb{S}^n \rightarrow \mathbb{S}^{n+m}$ related to Riccati differential equations \eqref{pf:gtrde} as : 
\begin{equation}
\label{operator_lambda}
\begin{aligned}
    &\begin{bmatrix} \frac{d}{dt}X(t) + A_0(t)^{\top}X+ XA_0(t) + \sum_{k=1}^r A_k(t)^{\top} XA_k(t) +M(t) & XB_0(t) + \sum_{k=1}^r A_k^{\top}(t)XB_k(t) + L(t) \\ B^{\top}_0(t)X + \sum_{k=1}^r B^{\top}_k(t)XA_k(t) + L^{\top}(t) & R(t)+ \sum_{k=1}^{r}B_{k}^{\top}(t)XB_{k}(t) \end{bmatrix}.  
\end{aligned}
\end{equation}
Define the set $\boldsymbol{\Gamma}^{\Sigma}$ related to Riccati differential equations \eqref{pf:gtrde}:
\begin{equation}
\label{set_gamma}
    \boldsymbol{\Gamma}^{\Sigma}=\left\{ X(\cdot) \in \mathbf{C}_b^1(\mathbb{R}_{+}, \mathbb{S}^n) \mid \boldsymbol{\Lambda}(t,X_t) \succeq 0,\, R(t)+ \sum_{k=1}^{r}B_{k}^{\top}(t)X(t)B_{k}(t) \succ 0 \,\, \text{for all} \,\, t \in \mathbb{R}_{+}\right\}.
\end{equation}

\begin{proposition}
\label{main_results:stabilizable_detectable}
When the parameters associated with the Riccati differential equations \eqref{pf:gtrde} satisfy the following conditions:
\begin{itemize}
    \item $R(t) \succ 0$ for all $t \in \mathbb{R}_{+}$;
    \item The system $[ A_{0,(0)}(\cdot), \dotsb, A_{r,(0)}(\cdot); B_{0}(\cdot), \dotsb, B_{r}(\cdot)]$ is \textit{stochastically stabilizable};
    \item There exists a continuous matrix valued functions set $\{C_0(\cdot),  \dots, C_r(\cdot)\}$ satisfying $\sum_{k=0}^{r}C_{k}^{\top}(t)C_{k}(t) = M(t)-L(t)R(t)L^{\top}(t) $ for all $t \in \mathbb{R}_{+}$ and the system $[C_{0}(\cdot), \dotsb, C_{r}(\cdot); A_{0,(0)}(\cdot), \dotsb, A_{r,(0)}(\cdot)]$ is \textit{stochastically detectable}.
\end{itemize}
Then the Riccati differential equations \eqref{pf:gtrde} has a unique stabilizing solution $X(\cdot)$  such that \[ X(t) \in \overline{\mathbb{S}}_{+}^n \,\, \text{and} \,\, R(t)+ \sum_{k=1}^{r}B_{k}^{\top}(t)X(t)B_{k}(t) \succ 0 \,\, \text{for all} \,\, t \in \mathbb{R}_+ .\] 
\end{proposition}

\begin{proof}
Based on Theorem 4.7 in \cite{Drgan2004}, together with the facts that $0 \in \boldsymbol{\Gamma}^{\Sigma}$ and the system \[[ A_{0,(0)}(\cdot), \dotsb, A_{r,(0)}(\cdot); B_{0}(\cdot), \dotsb, B_{r}(\cdot)]\] is \textit{stochastically stabilizable}, it follows that \eqref{pf:gtrde} admits a maximal solution $X^{\text{max}}(t) \in \overline{\mathbb{S}}_{+}^n$ such that $R(t)+ \sum_{k=1}^{r}B_{k}^{\top}(t)X(t)B_{k}(t) \succ 0$ for all $t \in \mathbb{R}_+$. By using Proposition \ref{function_g:proposition_3}, Riccati differential equations \eqref{pf:gtrde} satisfied by $X^{\text{max}}(\cdot)$ can be transformed into the following form:
\begin{equation*}
    \begin{aligned}
         &\frac{d}{dt}X^{\text{max}}(t)+ (A_{0}(t)+B_0(t)F(t,X^{\text{max}}_t))^{\top} X^{\text{max}}(t)+X^{\text{max}}(t)(A_{0}(t)+B_0(t)F(t,X^{\text{max}}_t)) \\
         &+ \sum_{k=1}^{r}(A_{r}(t)+B_k(t)F(t,X^{\text{max}}_t))^{\top} X^{\text{max}}(t)(A_k(t)+B_k(t)F(t,X^{\text{max}}_t)) +\sum_{k=0}^{r}\hat{C}_{k}^{\top}(t)\hat{C}_{k}(t)=0,t \ge 0,
    \end{aligned}
\end{equation*}
where
\[
\hat{C}_{k}(t)=\begin{pmatrix} \mathbb{O}_{q \times n}^{\times (k)} \\ C_{k}(t) \\\mathbb{O}_{q \times n}^{\times (r-k-1)} \\\mathbb{O}_{m \times n}^{\times (k)} \\\sqrt{\frac{1}{r+1}R(t)}(F(t,X^{\text{max}}_t)-F(t,0)) \\ \mathbb{O}_{m \times n}^{\times (r-k-1)}\end{pmatrix} \in \mathbf{R}^{[(q+m)(r+1)]\times n}, \forall \,t \in \mathbb{R}_+.
\]
Since the system $[C_{0,(0)}(\cdot), \dotsb, C_{r,(0)}(\cdot); A_{0,(0)}(\cdot), \dotsb, A_{r,(0)}(\cdot)]$ is \textit{stochastically detectable}, there exist a $\theta$-periodic continuous function $\varTheta(\cdot) : \mathbb{R}_+ \to \mathbb{R}^{n \times p}$ such that the system $( A_{0,(0)}(\cdot)+\varTheta(\cdot)C_{0}(\cdot), \dotsb, A_{r,(0)}(\cdot)+\varTheta(\cdot)C_{r}(\cdot))$ is stable.

Let 
\[
\hat{\varTheta}(t) = \begin{pmatrix} \varTheta(t)^{\times (r+1)} &-B_0(t)\sqrt{(r+1)R(t)^{-1}} &  \dotsb & -B_{r}(t)\sqrt{(r+1)R(t)^{-1}}\end{pmatrix} \in \mathbf{R}^{n\times [(q+m)(r+1)]}, \forall \,t \in \mathbb{R}_+,
\] 
we derive that the system  
\[
(A_0(\cdot) + B_0(\cdot)F((\cdot),X^{\text{max}}(\cdot))+\hat{\varTheta}(\cdot)\hat{C}_{0}(\cdot),\dotsb,A_r + B_r(\cdot)F((\cdot),X^{\text{max}}(\cdot))+\hat{\varTheta}(\cdot)\hat{C}_{r}(\cdot))
\] 
is stable. By Lemma \ref{main_results:stochastically_detectable_lemma} , it follows that the system 
\[
(A_0(\cdot) + B_0(\cdot)F((\cdot),X^{\text{max}}(\cdot)),\dotsb, A_r(\cdot) + B_r(\cdot)F((\cdot),X^{\text{max}}(\cdot)))
\]
is \textit{stochastically stabilizable}.
So $X^{\text{max}}(\cdot) $ is a stabilizing solution of the Riccati differential equations \eqref{pf:gtrde}. Given that there can be at most one stabilizing solution to the Riccati differential equations \eqref{pf:gtrde} ref to Theorem 5.6.5 and Corollary 5.6.7 of \cite{Dragan2013book}, it follows that the Riccati differential equations \eqref{pf:gtrde} has a unique stabilizing solution $X(\cdot)$  such that \[ X(t) \in \overline{\mathbb{S}}_{+}^n \,\, \text{and} \,\, R(t)+ \sum_{k=1}^{r}B_{k}^{\top}(t)X(t)B_{k}(t) \succ 0 \,\, \text{for all} \,\, t \in \mathbb{R}_+ .\]
\end{proof}

\subsection{Convergence Analysis of Iterative Sequence}

Next, we introduce two linear operators linked to interrelated iterative steps, whose specific properties are analyzed in Lemma \ref{main_results:lemma}. These favorable properties are critical for proving the algorithm's iterative sequence convergence, as explicitly shown in the proof of Theorem \ref{main_results:theorem}.  

For each triple $(K(\cdot), W(\cdot), X(\cdot), Z(\cdot))$ satisfying  $(K(\cdot),W(\cdot)) \in \mathcal{A}^\Sigma$ and $(t,X_t),(t,X_t+Z_t)\in\mathrm{Dom}\,\mathcal{G}$ for all $t \in \mathbb{R}_+$, the operator-valued functions $\mathcal{L}^{*}_{J_{X}}:\mathbb{R}_+ \to \mathbf{B}(\mathbb{S}^n)$ and $\mathcal{L}^{*}_{J_{X+Z}}:\mathbb{R}_+ \to \mathbf{B}(\mathbb{S}^n)$ are defined as follows:
\begin{equation}
\label{linear_operators_xkw}   
\begin{aligned}
    &\mathcal{L}^{*}_{J_{X}}(t)S = \left( A_{0,(0)}(t)+B_0(t)J_{X}(t) \right)^{\top}S + S \left( A_{0,(0)}(t)+B_0(t)J_{X}(t)\right) \\
    &+ \sum_{k=1}^r \left(A_{k,(0)}(t)+B_k(t) J_{X}(t)\right)^{\top} S \left(A_{k,(0)}(t)+B_k(t) J_{X}(t)\right)  \, , \forall \,S \in \mathbb{S}^n,t \in \mathbb{R}_+,
\end{aligned}
\end{equation}
\begin{equation}
\label{linear_operators_xzkw}
\begin{aligned}
    &\mathcal{L}^{*}_{J_{X+Z}}(t)S = \left( A_{0,(0)}(t)+B_0(t)J_{X+Z}(t) \right)^{\top}S + S \left( A_{0,(0)}(t)+B_0(t)J_{X+Z}(t)\right) \\
    &+ \sum_{k=1}^r \left(A_{k,(0)}(t)+B_k(t) J_{X+Z}(t)\right)^{\top} S \left(A_{k,(0)}(t)+B_k(t) J_{X+Z}(t)\right)  \, , \forall \,S \in \mathbb{S}^n,t \in \mathbb{R}_+,
\end{aligned}
\end{equation}
\begin{equation}
\label{linear_operators_auxiliary}
\begin{aligned}
\text{where}\quad
J_{X}(t)=&\begin{bmatrix} \mathbb{I}_{m_1}&\mathbb{O}_{m_2}\\W(t) &\mathbb{O}_{m_2}\end{bmatrix}\hat{F}(t,X_t)+\begin{bmatrix} \mathbb{O}_{m_1}\\K(t)\end{bmatrix}, &J_{X+Z}(t)=&\begin{bmatrix} \mathbb{I}_{m_1}&\mathbb{O}_{m_2}\\W(t) &\mathbb{O}_{m_2}\end{bmatrix}\hat{F}(t,X_t+Z_t)+\begin{bmatrix} \mathbb{O}_{m_1}\\K(t)\end{bmatrix},\text{and}\\
\quad\hat{F}(t,X_t)=&F(t,X_t)-F(t,0),\, \forall t \in \mathbb{R}_+
\end{aligned}    
\end{equation}

\begin{lemma}
\label{main_results:lemma}
Assume that $R_{22}(t) \succ 0$ for all $t \in \mathbb{R}_+$ and there exists $(K(\cdot), W(\cdot)) \in \mathcal{A}^\Sigma$. Let $\tilde{Y}_{KW}(\cdot)$ denote the $\theta$-periodic and stabilizing solution to the corresponding Riccati differential equation \eqref{pr:sgrde_kw}, and let $X(\cdot) ,Z(\cdot)\in \mathbf{C}^1(\mathbb{R}_+,\mathbb{S}^n)$ are $\theta$-periodic and satisfy the following properties:
\begin{itemize}
    \item $(t,X_t),(t,X_t+Z_t) \in \mathrm{Dom}\,\mathcal{G}$ for all $t \in \mathbb{R}_+.$
    \item $(t,X_t)$ and $(t,X_t+Z_t)$ satisfy:
    \begin{equation}
    \label{main_results:g_interrelated_equation}
    \begin{aligned}
        &\frac{d}{dt} (X(t)+ Z(t)) +(A_{0,(0)}(t)+ B_{0}(t)\hat{F}(t,X_t))^{\top}Z(t)+ Z(t)(A_{0,(0)}(t)+ B_{0}(t)\hat{F}(t,X_t))\\
        &+\sum_{k=1}^{r}(A_{k,(0)}(t)+ B_{k}(t)\hat{F}(t,X_t))^{\top}Z(t)(A_{k,(0)}(t)+ B_{k}(t)\hat{F}(t,X_t))+\mathcal{G}(t,X_t)-\\
        &\left[Z(t)B_{02}(t) + \sum_{k=1}^{r}(A_{k,(0)}(t)+ B_{k}(t)\hat{F}(t,X_t))^{\top}Z(t)B_{k2}(t)\right]\left[R_{22}(t) + \sum_{k=1}^{r}B_{k2}^\top(t) (X(t) + Z(t)) B_{k2}(t) \right]^{-1} \\
        &\quad \times \left[B_{02}^{\top}(t)Z(t) + \sum_{k=1}^{r}B_{k2}^{\top}(t)Z(t)(A_{k,(0)}(t)+ B_{k}(t)\hat{F}(t,X_t))\right]=0,t \ge 0,
    \end{aligned}
    \end{equation}
\end{itemize}
then we have:
\item [(i.)] If the system $(A_{0,(0)}(\cdot)+B_{0}(\cdot)J_{X}(\cdot),\dotsb,A_{r,(0)}+B_{r}(\cdot)J_{X}(\cdot))$ is stable, then $\tilde{Y}_{KW}(t)\succeq X(t) + Z(t)$ for all $t \in \mathbb{R}_+$.
\item [(ii.)] If $\tilde{Y}_{KW}(t) \succeq X(t) + Z(t)$ for all $t \in \mathbb{R}_+$, then the system 
\[ 
(A_{0,(0)}(\cdot)+B_{0}(\cdot)J_{X+Z}(\cdot),\dotsb,A_{r,(0)}(\cdot)+B_{r}(\cdot)J_{X+Z}(\cdot))
\]
is stable.
\end{lemma}

\begin{proof}
First, we prove the (i.). From Proposition \ref{function_g:proposition_2} we deduce 
\begin{align*}
\mathcal{G}&(t, X_t + Z_t) = \mathcal{L}^{*}_{J_{X}}(t)(X(t) + Z(t)) +\hat{F}^{\top}_1(t, X_t)L^{\top}_{KW}(t)+\hat{F}_1(t, X_t)L_{KW}(t)+ \hat{F}_1^{\top}(t, X_t) R_{W}(t)\hat{F}_1(t, X_t)\\
&+  M_K(t) - \left[ J_{X}(t)- \hat{F}(t, X_t+Z_t) \right]^\top \left[ R(t) + \sum_{k=1}^{r}B_k^\top(t) (X(t) + Z(t)) B_k(t) \right] \left[ J_{X}(t)- \hat{F}(t, X_t+Z_t) \right],
\end{align*}
where $\hat{F}_1(t, X_t)=\begin{bmatrix} \mathbb{I}_{m_1} & \mathbb{O}_{m_2} \end{bmatrix}\hat{F}(t, X_t)$, $\mathcal{L}^{*}_{J_{X}}(\cdot)$ is defined in \eqref{linear_operators_xkw}, $J_{X}(\cdot)$ and $\hat{F}(\cdot, \cdot)$ is defined in \eqref{linear_operators_auxiliary}.
Combining \eqnref{function_g:proposition_1_xz} and \eqnref{main_results:g_interrelated_equation}, we obtain that $X(\cdot) + Z(\cdot)$ solves following differential equation:
\begin{equation}
\label{lemma_proof_1}
\begin{aligned}
\frac{d}{dt}& (X(t) + Z(t)) + \mathcal{L}^{*}_{J_{X}}(t)(X(t) + Z(t)) + \hat{F}^{\top}_1(t, X_t)L^{\top}_{KW}(t)+ \hat{F}_1(t, X_t)L_{KW}(t)+ \hat{F}_1^{\top}(t, X_t) R_{W}(t)\hat{F}_1(t, X_t)\\
&+ M_K(t)- \left[ J_{X}(t)- \hat{F}(t, X_t+Z_t) \right]^\top \left[ R(t) + \sum_{k=1}^{r}B_k^\top(t) (X(t) + Z(t)) B_k(t) \right] \left[ J_{X}(t)- \hat{F}(t, X_t+Z_t) \right] \\
&+ \left[ N_1(t, X_t, Z_t)-R_{12}(t,X_t + Z_t)R_{22}(t,X_t + Z_t)^{-1}N_2(t, X_t,Z_t) \right]^{\top} \mathbb{R}_{22}^{\sharp}(t, X_t + Z_t)^{-1}\\
& \quad \times \left[ N_1(t, X_t, Z_t)-R_{12}(t,X_t + Z_t)R_{22}(t,X_t + Z_t)^{-1}N_2(t, X_t,Z_t) \right]=0,t \geq 0,
\end{aligned}     
\end{equation}
where $\mathbb{R}_{22}^{\sharp}(t, X_t + Z_t)$ is defined in \eqref{pf:sign_conditions_1e} and 
\[
\begin{bmatrix} N_1(t, X_t,Z_t) \\ N_2(t, X_t,Z_t)  \end{bmatrix}=\begin{bmatrix} B_{01}^{\top}(t)Z(t) + \sum_{k=1}^{r}B_{k1}^{\top}(t)Z(t)(A_{k,(0)}(t)+ B_{k}(t)\hat{F}(t,X_t))\\ B_{02}^{\top}(t)Z(t) + \sum_{k=1}^{r}B_{k2}^{\top}(t)Z(t)(A_{k,(0)}(t)+ B_{k}(t)\hat{F}(t,X_t)) \end{bmatrix}\,\,\text{for all}\,\,t \in \mathbb{R}_+.
\]
Moreover,
\begin{equation}
\label{lemma_proof_2}
\begin{aligned}
& \left[ J_{X}(t)- \hat{F}(t, X_t+Z_t) \right]^\top \left[ R(t) + \sum_{k=1}^{r}B_k^\top(t) (X(t) + Z(t)) B_k(t) \right] \left[ J_{X}(t)- \hat{F}(t, X_t+Z_t) \right]\\
=& \left[\hat{F}_1(t, X_t)- \hat{F}_1(t, X_t + Z_t)\right]^\top \mathbb{R}_{22}^{\sharp}(t, X_t + Z_t)\left[\hat{F}_1(t, X_t)- \hat{F}_1(t, X_t + Z_t)\right]+H^{\top}_1(t)H_1(t)\\
=& \left[ N_1(t, X_t, Z_t)-R_{12}(t,X_t + Z_t)R_{22}(t,X_t + Z_t)^{-1}N_2(t, X_t,Z_t) \right]^{\top} \mathbb{R}_{22}^{\sharp}(t, X_t + Z_t)^{-1}\\
& \quad \times \left[ N_1(t, X_t, Z_t)-R_{12}(t,X_t + Z_t)R_{22}(t,X_t + Z_t)^{-1}N_2(t, X_t,Z_t) \right]+H_1^{\top}(t)H_1(t),
\end{aligned}
\end{equation}
where $H_1(t)=\begin{bmatrix}R_{22}(t,X_t + Z_t)^{-\frac{1}{2}}R_{21}(t,X_t + Z_t) & R_{22}(t,X_t + Z_t)^{\frac{1}{2}}\end{bmatrix} \left[ J_{X}(t)- \hat{F}(t, X_t + Z_t) \right]$.
Next, by virtue of Proposition \ref{function_g:proposition_2}, $\tilde{Y}_{KW}(\cdot)$ may be rewritten in the form:
\begin{equation}
\label{lemma_proof_3}
\begin{aligned}
&\frac{d}{dt} \tilde{Y}_{KW}(t) + \mathcal{L}^{*}_{J_{X}}(t)\tilde{Y}_{KW}(t)+ M_K(t) +\hat{F}_1(t, X_t)^{\top} R_{W}(t)\hat{F}_1(t, X_t)+L_{KW}^{\top}(t)\hat{F}_1^{\top}(t, X_t)+L_{KW}(t)\hat{F}_1(t, X_t)\\
&- \left[ \hat{F}_1(t, X_t) - \tilde{F}_{KW}(t) \right]^\top \left[ R_W(t) + \sum_{k=1}^r B_{kW}^{\top}(t)\tilde{Y}_{KW}(t)B_{kW}(t)\right] \left[ \hat{F}_1(t, X_t) - \tilde{F}_{KW}(t) \right]= 0,t \geq 0,
\end{aligned} 
\end{equation}
where $\tilde{F}_{KW}(t)=$
\[
-\left[R_W(t) + \sum_{k=1}^r B_{kW}^{\top}(t)\tilde{Y}_{KW}(t)B_{kW}(t)\right]^{-1}\left[B^{\top}_{0W}(t)\tilde{Y}_{KW}(t) + \sum_{k=1}^r B^{\top}_{kW}(t)\tilde{Y}_{KW}(t)A_{kK}(t) + L^{\top}_{KW}(t)\right],\forall t \in \mathbb{R}_+
\]
Subtracting \eqnref{lemma_proof_1} from \eqnref{lemma_proof_3} and combining it with \eqnref{lemma_proof_2} yields:
\begin{align*}
    &\frac{d}{dt} (\tilde{Y}_{KW}(t)-X(t) - Z(t)) + \mathcal{L}^{*}_{J_{X}}(t)(\tilde{Y}_{KW}(t)-X(t) - Z(t)) +\hat{H_1}^{\top}(t)\hat{H_1}(t),t \geq 0,\quad \hat{H_1}^{\top}(t)\hat{H_1}(t)=\\
    &H_1^{\top}(t)H_1(t)- \left[ \hat{F}_1(t, X_t) - \tilde{F}_{KW}(t) \right]^\top \left[ R_W(t) + \sum_{k=1}^r B_{kW}^{\top}(t)\tilde{Y}_{KW}(t)B_{kW}(t)\right] \left[ \hat{F}_1(t, X_t) - \tilde{F}_{KW}(t) \right].   
\end{align*}
Since the system $(A_{0,(0)}(\cdot)+B_{0}(\cdot)J_{X}(\cdot),\dotsb,A_{r,(0)}(\cdot)+B_{r}(\cdot)J_{X}(\cdot))$ is stable and $\hat{H_1}^{\top}(t)\hat{H_1}(t)\succeq0$ for all $t \in \mathbb{R}_+$, by Lemma \ref{main_results:stochastically_detectable_lemma}, we have $\tilde{Y}_{KW}(t) \succeq X(t) + Z(t)$ for all $t \in \mathbb{R}_+$. This completes the proof of part (i.).

Now, we proceed to prove the (ii.). Similarly, by Proposition \ref{function_g:proposition_2} and \ref{function_g:proposition_3}, and through combining \eqnref{main_results:function_g} with \eqnref{main_results:g_interrelated_equation}, we can derive $X(\cdot) + Z(\cdot)$ solves following differential equation:
\begin{equation}
\label{lemma_proof_4}
\begin{aligned}
&\frac{d}{dt} (X(t) + Z(t)) + \mathcal{L}^{*}_{J_{X+Z}}(t)(X(t) + Z(t)) + M_K(t)\\
&+\hat{F}^{\top}_1(t, X_t+Z_t)L^{\top}_{KW}(t)+\hat{F}_1(t, X_t+Z_t)L_{KW}(t)+ \hat{F}_1(t, X_t+Z_t)^{\top} R_{W}(t)\hat{F}_1(t, X_t+Z_t)\\
&- \left[ J_{X+Z}(t)- \hat{F}(t, X_t + Z_t) \right]^\top \left[ R(t) + \sum_{k=1}^{r}B_k^\top(t) (X(t) + Z(t)) B_k(t) \right] \left[ J_{X+Z}(t)- \hat{F}(t, X_t + Z_t) \right]\\
&+ \left[ N_1(t, X_t, Z_t)-R_{12}(t,X_t + Z_t)R_{22}(t,X_t + Z_t)^{-1}N_2(t, X_t,Z_t) \right]^{\top} \mathbb{R}_{22}^{\sharp}(t, X_t + Z_t)^{-1}\\
& \quad \times \left[ N_1(t, X_t, Z_t)-R_{12}(t,X_t + Z_t)R_{22}(t,X_t + Z_t)^{-1}N_2(t, X_t,Z_t) \right]=0,t \geq 0,
\end{aligned}     
\end{equation}
where $\mathcal{L}^{*}_{J_{X+Z}}$ is defined in \eqref{linear_operators_xzkw}.
And 
\begin{equation}
\label{lemma_proof_5}
\begin{aligned}
&\left[ J_{X+Z}(t)- \hat{F}(t, X_t + Z_t) \right]^\top \left[ R(t) + \sum_{k=1}^{r}B_k^\top(t) (X(t) + Z(t)) B_k(t) \right] \left[ J_{X+Z}(t)- \hat{F}(t, X_t + Z_t) \right]\\
=& \left[\hat{F}_1(t, X_t + Z_t)- \hat{F}_1(t, X_t + Z_t)\right]^\top \mathbb{R}_{22}^{\sharp}(t, X_t + Z_t)^{-1}\left[\hat{F}_1(t, X_t + Z_t)- \hat{F}_1(t, X_t + Z_t)\right]+H_2^{\top}(t)H_2(t)\\
=&H_2^{\top}(t)H_2(t),
\end{aligned}    
\end{equation}
where $H_2(t)=\begin{bmatrix}R_{22}(t,X_t + Z_t)^{-\frac{1}{2}}R_{21}(t,X_t + Z_t) & R_{22}(t,X_t + Z_t)^{\frac{1}{2}}\end{bmatrix} \left[ J_{X+Z}(t)- \hat{F}(t, X_t + Z_t) \right] $.
Also, $\tilde{Y}_{KW}(\cdot)$ may be rewritten in the form:
\begin{equation}
\label{lemma_proof_6}
\begin{aligned}
&\frac{d}{dt} \tilde{Y}_{KW}(t) + \mathcal{L}^{*}_{J_{X+Z}}(t)\tilde{Y}_{KW}(t)+ M_K(t) \\
&+\hat{F}_1(t, X_t+Z_t)^{\top} R_{W}(t)\hat{F}_1(t, X_t+Z_t)+L_{KW}^{\top}(t)\hat{F}_1^{\top}(t, X_t+Z_t)+L_{KW}(t)\hat{F}_1(t, X_t+Z_t)\\
&- \left[ \hat{F}_1(t, X_t+Z_t) - \tilde{F}_{KW}(t) \right]^\top \left[ R_W(t) + \sum_{k=1}^r B_{kW}^{\top}(t)\tilde{Y}_{KW}(t)B_{kW}(t)\right] \left[ \hat{F}_1(t, X_t+Z_t) - \tilde{F}_{KW}(t) \right] = 0,t \geq 0.
\end{aligned} 
\end{equation}
Subtracting \eqnref{lemma_proof_4} from \eqnref{lemma_proof_6} combining it with \eqnref{lemma_proof_5} yields:
\begin{align*}
    &\frac{d}{dt} (\tilde{Y}_{KW}(t)-X(t) - Z(t)) + \mathcal{L}^{*}_{J_{X+Z}}(t)(\tilde{Y}_{KW}(t)-X(t) - Z(t)) + H_2^{\top}(t)H_2(t)\\
    &-\left[ \hat{F}_1(t, X_t+Z_t) - \tilde{F}_{KW}(t) \right]^\top \left[ R_W(t) + \sum_{k=1}^r B_{kW}^{\top}(t)\tilde{Y}_{KW}(t)B_{kW}(t)\right] \left[ \hat{F}_1(t, X_t+Z_t) - \tilde{F}_{KW}(t) \right]\\
    &- \left[ N_1(t, X_t, Z_t)-R_{12}(t,X_t + Z_t)R_{22}(t,X_t + Z_t)^{-1}N_2(t, X_t,Z_t) \right]^{\top} \mathbb{R}_{22}^{\sharp}(t, X_t + Z_t)^{-1}\\
    & \quad \times \left[ N_1(t, X_t, Z_t)-R_{12}(t,X_t + Z_t)R_{22}(t,X_t + Z_t)^{-1}N_2(t, X_t,Z_t) \right]=0,t \geq 0.   
\end{align*}
Let $U(t)=\tilde{Y}_{KW}(t)-X(t) - Z(t)$, we have $U(t)\succeq 0$ for all $t \in \mathbb{R}_+$ and
\begin{align*}
    \frac{d}{dt} U(t)+\mathcal{L}^{*}_{J_{X+Z}}(t)U(t)+\sum_{k=0}^{r}\hat{E}_k^{\top}(t)\hat{E}_k(t)=0,t \geq 0,
\end{align*}
where $\hat{E}_k(t)\in\mathbf{R}^{[m_2 + m_1(r+2)]\times n}(k=0,\dotsb,r)$
\[
    \hat{E}_k(t)=\begin{pmatrix}  \frac{1}{\sqrt{r+1}}H_2(t) \\ \mathbb{O}_{m_1 \times n}^{\times (k)}\\\sqrt{-\frac{1}{r+1} \left[ R_W(t) + \sum_{k=1}^r B_{kW}^{\top}(t)\tilde{Y}_{KW}(t)B_{kW}(t)\right]} \left[ \hat{F}_1(t, X_t+Z_t) - \tilde{F}_{KW}(t) \right]\\\mathbb{O}_{m_1 \times n}^{\times (r-k)}\\ \sqrt{-\frac{1}{r+1}\mathbb{R}_{22}^{\sharp}(t, X_t + Z_t)^{-1}}\left[ N_1(t, X_t, Z_t)-R_{12}(t,X_t + Z_t)R_{22}(t,X_t + Z_t)^{-1}N_2(t, X_t,Z_t) \right]\end{pmatrix}.
\]
For all $t \in \mathbb{R}_+$, select $\hat{\varTheta}(t)\in\mathbf{R}^{n\times [m_2 + m_1(r+2)]}$ in the following form:
\[
\hat{\varTheta}(t)=\begin{pmatrix} \mathbb{O}_{n \times m_2}^{\top}\\ -\sqrt{-(r+1)\left[ R_W(t) + \sum_{k=1}^r B_{kW}^{\top}(t)\tilde{Y}_{KW}(t)B_{kW}(t)\right]^{-1}}^{\top}(B_{01}(t)+B_{02}(t)W(t))^{\top}\\ \dotsb \\-\sqrt{-(r+1)\left[ R_W(t) + \sum_{k=1}^r B_{kW}^{\top}(t)\tilde{Y}_{KW}(t)B_{kW}(t)\right]^{-1}}^{\top}(B_{r1}(t)+B_{r2}(t)W(t))^{\top} \\ \mathbb{O}_{n \times m_1}^{\top} \end{pmatrix}^{\top}.
\]
We obtain that 
\[
(A_{0,(0)}(\cdot)+B_{0}(\cdot)J_{X+Z}(\cdot)+\hat{\varTheta}(\cdot)\hat{E}_0(\cdot),\dotsb,A_{r,(0)}(\cdot)+B_{r}(\cdot)J_{X+Z}(\cdot)+\hat{\varTheta}(\cdot)\hat{E}_r(\cdot)) 
\]
is stable, which implies the system 
\[
[\hat{E}_0(\cdot),\dotsb,\hat{E}_r(\cdot);A_{0,(0)}(\cdot)+B_{0}(\cdot)J_{X+Z}(\cdot),\dotsb,A_{r,(0)}(\cdot)+B_{r}(\cdot)J_{X+Z}(\cdot)]
\] is \textit{stochastically detectable}.
By Lemma \ref{main_results:stochastically_detectable_lemma}, we have the system $(A_{0,(0)}(\cdot)+B_{0}(\cdot)J_{X+Z}(\cdot),\dotsb,A_{r,(0)}(\cdot)+B_{r}(\cdot)J_{X+Z}(\cdot))$ is stable. Thus completing the proof of (ii.).
\end{proof}

\begin{theorem}
\label{main_results:theorem}
Assume the following conditions hold:
\begin{itemize}
    \item $R_{22}(t) \succ 0$ for all $t \in \mathbb{R}_+$.
    \item The set $\mathcal{A}^{\Sigma}$ is non-empty.
    \item There exists a continuous matrix valued functions set $\{E_0(\cdot),  \dots, E_r(\cdot)\}$ satisfying $\sum_{k=0}^{r}E_{k}^{\top}(t)E_{k}(t) = M(t)-L(t)R(t)L^{\top}(t) $ for all $t \in \mathbb{R}_{+}$ and the system 
    \[
    \left[  E_{0,(0)}(\cdot),\ldots, E_{r,(0)}(\cdot); A_{0,(0)}(\cdot), \ldots, A_{r,(0)}(\cdot) \right]
    \]
    is \textit{stochastically detectable}.
\end{itemize}
Then, we have:
\item[1.] 
The sequences $\{Z^{(h)}(\cdot)\}_{h\geq0}$, $\{X^{(h)}(\cdot)\}_{h\geq0}$ are well defined by \eqref{algorithm:initial} \eqref{algorithm:external_circulation}  \eqref{algorithm:internal_circulation}, and for each $h = 0,1,2\ldots$ the following items are fulfilled:
\begin{enumerate}
    \item [$a_h$.] The system $\left( A_{0,(h)}(\cdot)+B_{02}(\cdot)T(\cdot,X^{(h)}(\cdot),Z^{(h)}(\cdot)),\dotsb,A_{r,(h)}(\cdot)+B_{r2}(\cdot)T(\cdot,X^{(h)}(\cdot),Z^{(h)}(\cdot)) \right)$ is stable, where $T(t,X^{(h)}_t,Z^{(h)}_t)(\forall t \in \mathbb{R}_+)=$
    \[
    -[ R_{22}(t)+\sum_{k=1}^{r}B^{\top}_{k2}(t)(X^{(h)}(t)+Z^{(h)}(t))B_{k2}(t) ]^{-1} [  B_{02}^{\top}(t)Z(t) + \sum_{k=1}^{r}B_{k2}^{\top}(t)Z(t)(A_{k,(0)}(t)+ B_{k}(t)\hat{F}(t,X_t)) ];
    \]
    \item [$b_h$.] Let $(K(\cdot), W(\cdot)) \in \mathcal{A}^\Sigma$ is arbitrary but fixed and $\tilde{Y}_{KW}(\cdot)$ be the $\theta$-periodic and stabilizing solution to the corresponding Riccati differential equation \eqref{pr:sgrde_kw}. Then we have $\tilde{Y}_{KW}(t) \succeq   X^{(h)}(t) + Z^{(h)}(t)$ for all $t \in \mathbb{R}_+ $;
    \item [$c_h$.] For all $t \in \mathbb{R}_+, (t, X^{(h)}_t),(t, X^{(h)}_t+ Z^{(h)}_t) \in \text{Dom } \mathcal{G}$ and
    \begin{align*}
    &\frac{d}{dt}(X^{(h)}(t)+Z^{(h)}(t))+ \mathcal{G}(t, X^{(h)}_t+Z^{(h)}_t)=-V_{(h)}^{\top}(t)\mathbb{R}_{22}^{\sharp}(t, X^{(h)}_t+ Z^{(h)}_t)^{-1} V_{(h)}(t);
    \end{align*}
    where $V_{(h)}(\cdot)$ is defined in \eqref{algorithm:vxz}.
    \item [$d_h$.] The system $(A_{0,(0)}(\cdot)+B_{0}(\cdot)J_{X^{(h)}+Z^{(h)}}(\cdot),\dotsb,A_{0,(0)}(\cdot)+B_{r}(\cdot)J_{X^{(h)}+Z^{(h)}}(\cdot))$ is stable,
    where $J_{X^{(h)}+Z^{(h)}}(\cdot)$ is defined in \eqref{linear_operators_auxiliary}.
\end{enumerate}
\item[2.] 
$\lim_{h\rightarrow\infty}X^{(h)}(t)=\tilde{X}(t),\forall t \in \mathbb{R}_+$, where $\tilde{X}(\cdot)$ is a unique $\theta$-periodic and stabilizing solution of the SGTRDE \eqref{pf:gtrde}.
\end{theorem}

\begin{proof}
\item[$\mathrm{i}$.]  
Let $X^{(0)}(\cdot) = 0 $ and $ Z^{(0)}(\cdot)$ is the $\theta$-periodic and stabilizing solution to the following Riccati differential equations:
\begin{equation}
\label{proof_0}
\begin{aligned}
    &\frac{d}{dt}Z(t)+ A^{\top}_{0,(0)}(t) Z(t)+Z(t)A_{0,(0)}(t)+\sum_{k=1}^{r}A^{\top}_{k,(0)}(t) Z(t)A_{k,(0)}(t) + M(t) -L(t) R(t)^{-1}L^{\top}(t)\\
    &- \left[Z(t)B_{02}(t) + \sum_{k=1}^{r} A^{\top}_{k,(0)}(t)Z(t)B_{k2}(t)\right]\left[R_{22}(t) + \sum_{k=1}^{r}B_{k2}(t)^{\top} Z(t) B_{k2}(t)\right]^{-1}\\
    & \quad \times \left[B^{\top}_{02}(t)Z(t) + \sum_{k=1}^{r} B^{\top}_{k2}(t)Z(t)A_{k,(0)}(t)\right]= 0,t \ge 0.     
\end{aligned}    
\end{equation}
To show that $Z^{(0)}(\cdot)$ is well-defined, we first prove the existence and unique of $\theta$-periodic and stabilizing solution to the above \eqnref{proof_0}. 
Since the set $\mathcal{A}^\Sigma$ is nonempty and set $L^{(0)}(\cdot) = K(\cdot)$, we can find the system \[(A_{0,(0)}(\cdot)+B_{02}(\cdot)L^{(0)}(\cdot), \dots, A_{r,(0)}(\cdot)+B_{r2}(\cdot)L^{(0)}(\cdot))\] associated with the system $(A_{0,(0)}(\cdot)+B_{02}(\cdot)K(\cdot),\dots, A_{r,(0)}(\cdot)+B_{r2}(\cdot)K(\cdot))$ is stable. This means the system $[A_{0,(0)}(\cdot),\dotsb,A_{r,(0)}(\cdot);B_{02}(\cdot),\dotsb,B_{r2}(\cdot)]$ is \textit{stochastically stabilizable}.
Since the system \[ \left[  E_{0,(0)}(\cdot),\ldots, E_{r,(0)}(\cdot); A_{0,(0)}(\cdot), \ldots, A_{r,(0)}(\cdot) \right]\] is \textit{stochastically detectable}. Under Assumption $R_{22}(t) \succ 0$ for all $t \in \mathbb{R}_+$, the \eqnref{proof_0} admits a unique $\theta$-periodic and stabilizing solution $Z^{(0)}(\cdot) $ satisfying $R_{22}(t,Z^{(0)}_t) \succ 0$ for all $t \in \mathbb{R}_+ $ by using Proposition \ref{main_results:stabilizable_detectable}. Therefore, $Z^{(0)}(\cdot) $ is well-defined as the $\theta$-periodic and stabilizing solution to the \eqnref{proof_0}. This also means the system $(A_{0,(0)}(\cdot)+B_{02}(\cdot)T(\cdot,X^{(0)}(\cdot),Z^{(0)}(\cdot)),\dotsb,A_{r,(0)}(\cdot)+B_{r2}(\cdot)T(\cdot,X^{(0)}(\cdot),Z^{(0)}(\cdot)))$ is stable.

Let $(K(\cdot), W(\cdot)) \in \mathcal{A}^\Sigma$ be arbitrary but fixed and $\tilde{Y}_{KW}(\cdot)$ denote the $\theta$-periodic and stabilizing solution to the corresponding Riccati differential equation \eqref{pr:sgrde_kw}. Since $(A_{0,(0)}(\cdot)+B_{0}(\cdot)J_{X^{(0)}}(\cdot),\dotsb,A_{r,(0)}(\cdot)+B_{0}(\cdot)J_{X^{(0)}}(\cdot))$ associated with the system $(A_{0,(0)}(\cdot),\dotsb,A_{r,(0)}(\cdot))$ is stable. From (i.) and (ii.) in Lemma \ref{main_results:lemma}, we get $\tilde{Y}_{KW}(t) \succeq   X^{(1)}(t) + Z^{(1)}(t)$, for all $t \in  \mathbb{R}_+$ and the system $(A_{0,(0)}(\cdot)+B_{0}(\cdot)J_{X^{(0)}+Z^{(0)}}(\cdot),\dotsb,A_{r,(0)}(\cdot)+B_{r}(\cdot)J_{X^{(0)}+Z^{(0)}}(\cdot))$ is stable.

By using Lemma \ref{pr:lemma}, we have  $(t, 0) \in \text{Dom}\,\mathcal{G}$ and $(t, \tilde{Y}_{KW}(t)) \in \text{Dom}\,\mathcal{G}$ for all $t \in \mathbb{R}_+$.
Since $R_{22}(t,Z_t^{(0)}) \\\succ 0$ and $\mathbb{R}(t, X_t^{(0)} + Z_t^{(0)}) \prec \mathbb{R}(t, \tilde{Y}_{KW}(t)) \prec 0$ for all $t \in \mathbb{R}_+$, we obtain that $\mathbb{R}_{22}^\sharp(P^{(0)} + Z^{(0)}) \prec \mathbb{R}_{22}^\sharp(\tilde{P}_{LW}) \prec 0$ by using Lemma \ref{pr:lemma} and Corollary 4.5 in \cite{freiling2003}. It follows that $P^{(0)}, P^{(0)} + Z^{(0)} \in \text{Dom } \mathcal{G}$.
Substituting \eqnref{algorithm:initial} into \eqnref{function_g:proposition_1_xz} then gives:
\begin{align*}
    &\frac{d}{dt}\left( X^{(0)}(t) + Z^{(0)}(t) \right) + \mathcal{G}\left( t, X_t^{(0)} + Z_t^{(0)} \right) = -V_{(0)}^\top(t) \mathbb{R}_{22}^\sharp\left( t, X_t^{(0)} + Z_t^{(0)} \right)^{-1} V_{(0)}(t).
\end{align*}
This completes the proof of statements $a_0$-$d_0$.

\item[$\mathrm{ii}$.]
Assume $h=l-1$, $Z^{(l-1)}(\cdot)$ is well-defined and that $a_{l-1}-d_{l-1}$ hold.  We shall prove that  $Z^{(l)}(\cdot)$ is well-defined and $a_{l}-d_{l}$ hold.

For $h = l $ ,  $Z^{(l)}(\cdot) $ satisfies the following Riccati differential equations:
\begin{equation}
\label{proof_l}
\begin{aligned}
    &\frac{d}{dt}(X^{(l)}(t)+Z(t))+ A^{\top}_{0,(l)}(t) Z(t)+Z(t)A_{0,(l)}(t)+\sum_{k=1}^{r}A^{\top}_{k,(l)}(t) Z(t)A_{k,(l)}(t) +\sum_{k=0}^{r}E_{k,(l)}^{\top}(t)E_{k,(l)}(t)\\
    &- \left[Z(t)B_{02}(t) + \sum_{k=1}^{r} A^{\top}_{k,(l)}(t)Z(t)B_{k2}(t)\right]\left[R_{22}(t) + \sum_{k=1}^{r}B_{k2}(t)^{\top} Z(t) B_{k2}(t)\right]^{-1}\\
    & \quad \times \left[B^{\top}_{02}(t)Z(t) + \sum_{k=1}^{r} B^{\top}_{k2}(t)Z(t)A_{k,(l)}(t)\right]= 0,t \ge 0.     
\end{aligned}
\end{equation}
where $(k=0,\dots,r)$
\[
    E_{k,(l)}(t)=\begin{pmatrix} \mathbb{O}_{m_1 \times n}^{\times (k)} \\ \sqrt{-\frac{1}{r+1}\mathbb{R}_{22}^\sharp\left( t, X_t^{(l-1)} + Z_t^{(l-1)} \right)^{-1} }V_{(l-1)}(t) \\ \mathbb{O}_{m_1 \times n}^{\times (r-k)}\end{pmatrix} \in \mathbb{R}^{n\times [m_1(r+1)]}.
\]

We demonstrate that $Z^{(l)}(\cdot)$ is well-defined. By the Proposition \ref{main_results:stabilizable_detectable} , proving that $Z^{(l)}(\cdot)$ is well-defined is equivalent to proving that the system $[A_{0,(l)}(\cdot),\dots,A_{r,(l)}(\cdot); B_{02}(\cdot),\dots,B_{r2}(\cdot)]$ is \textit{stochastically stabilizable} and $[E_{0,(l)}(\cdot),\dots,E_{r,(l)}(\cdot); A_{0,(l)}(\cdot),\dots,A_{r,(l)}(\cdot)] $ is \textit{stochastically detectable}. To prove the stability of the system $[A_{0,(l)}(\cdot),\dots,A_{r,(l)}(\cdot); B_{02}(\cdot),\dots,B_{r2}(\cdot)]$, it suffices to set  \[L^{(l)}(t) = K(t)+\begin{pmatrix} W(t) & -\mathbb{I}_{m_2} \end{pmatrix}  \hat{F}(t,X^{(l-1)}_t+Z^{(l-1)}_t), \,\,\text{for all}\,\, t \in \mathbb{R}_+. \] Then we can find the system $(A_{0,(l)}(\cdot)+B_{02}L^{(l)}(\cdot), \dots, A_{r,(l)}(\cdot)+B_{02}L^{(l)}(\cdot))$ associated with the system $(A_{0,(0)}(\cdot)+B_{0}(\cdot)J_{X^{(l-1)}+Z^{(l-1)}}(\cdot),\dotsb,A_{r,(0)}(\cdot)+B_{r}(\cdot)J_{X^{(l-1)}+Z^{(l-1)}}(\cdot))$ is stable. This means the system $[A_{0,(l)}(\cdot),\\\dots,A_{r,(l)}(\cdot); B_{02}(\cdot),\dots,B_{r2}(\cdot)]$ \textit{stochastically stabilizable}.

Let 
\[
\tilde{\varTheta}^{(l)}(t) =\begin{pmatrix} \tilde{\varTheta}_{0}^{(l)}(t) & \dots  & \tilde{\varTheta}_{r}^{(l)}(t)\end{pmatrix}\in \mathbb{R}^{n\times m_1(r+1)}, \forall t \in \mathbb{R}_+,
\]
where $(k=0,\cdots,r)$
\[
\tilde{\varTheta}_{k}^{(l)}(t)= \begin{pmatrix}  -\sqrt{(r+1)}(B_{k1}(t)-B_{k2}(t)R_{22}(t,X^{(l)}_t )^{-1}R_{21}(t,X^{(l)}_t ))\sqrt{-\mathbb{R}_{22}^\sharp\left( t, X_t^{(l)}\right)^{-1} } \end{pmatrix}
\]
Then we can find the system $(A_{0,(l)}(\cdot)+\tilde{\varTheta}^{(l)}(\cdot) E_{0,(l)}(\cdot),\dots, A_{r,(l)}(\cdot)+\tilde{\varTheta}^{(l)}(\cdot) E_{r,(l)}(\cdot))$ associated with the system $(A_{0,(l)}(\cdot)+B_{02}(\cdot)T(\cdot,X^{(l-1)}(\cdot),Z^{(l-1)}(\cdot)),\dotsb,A_{r,(h)}(\cdot)+B_{r2}(\cdot)T(\cdot,X^{(l-1)}(\cdot),Z^{(l-1)}(\cdot)))$ is stable. This means the system $[E_{0,(l)}(\cdot),\dots,E_{r,(l)}(\cdot); A_{0,(l)}(\cdot),\dots,A_{r,(l)}(\cdot)] $ is \textit{stochastically detectable}. From Proposition \ref{main_results:stabilizable_detectable}, the \eqnref{proof_l} admits a unique $\theta$-periodic and stabilizing solution $Z^{(l)}(\cdot) $ satisfying $R_{22}(t,X_t^{(l)} + Z_t^{(l)}) \succ 0$ for all $t \in \mathbb{R}_+ $. Therefore, $Z^{(l)}(\cdot) $ is well-defined as the $\theta$-periodic and stabilizing solution to the \eqnref{proof_l} and the system $[A_{0,(1)}(\cdot)+B_{02}(\cdot)T(\cdot,X^{(l)}(\cdot),Z^{(l)}(\cdot)),\dotsb,A_{r,(1)}(\cdot)+B_{r2}(\cdot)T(\cdot,X^{(l)}(\cdot),Z^{(l)}(\cdot))]$ is stable. 

Since $(A_{0}(\cdot)+B_{0}(\cdot)J_{X^{(l-1)}+Z^{(l-1)}}(\cdot),\dotsb,A_{r}(\cdot)+B_{r}(\cdot)J_{X^{(l-1)}+Z^{(l-1)}}(\cdot))$ is stable. From (i.) and (ii.) in Lemma \ref{main_results:lemma}, we get $\tilde{Y}_{KW}(t) \succeq   X^{(l)}(t) + Z^{(l)}(t)$ for all $t \in  \mathbb{R}_+$ and the system 
\[
(A_{0}(\cdot)+B_{0}(\cdot)J_{X^{(l)}+Z^{(l)}}(\cdot),\dotsb,A_{r}(\cdot)+B_{r}(\cdot)J_{X^{(l)}+Z^{(l)}}(\cdot))
\] is stable.

Since $R_{22}(t,X_t^{(l)} + Z_t^{(l)}) \succ 0$ and $\mathbb{R}(t, X_t^{(l)} + Z_t^{(l)}) \prec \mathbb{R}(t, \tilde{Y}_{KW}(t)) \prec 0$ for all $t \in \mathbb{R}_+$, we obtain that $\mathbb{R}_{22}^\sharp(P^{(l)} + Z^{(l)}) \prec \mathbb{R}_{22}^\sharp(\tilde{P}_{LW}) \prec 0$ by using Lemma \ref{pr:lemma} and Corollary 4.5 in \cite{freiling2003}. It follows that $P^{(l)}, P^{(l)} + Z^{(l)} \in \text{Dom } \mathcal{G}$.
Also, substituting \eqnref{proof_l} into \eqnref{function_g:proposition_1_xz}, we obtain:
\begin{align*}
    &\frac{d}{dt}\left( X^{(l)}(t) + Z^{(l)}(t) \right) + \mathcal{G}\left( t, X_t^{(l)} + Z_t^{(l)} \right) = -V_{(l)}^\top(t) \mathbb{R}_{22}^\sharp\left( t, X_t^{(l)} + Z_t^{(l)} \right)^{-1} V_{(l)}(t).
\end{align*}
Thus, we have proved the statements $a_{l}-d_{l}$.

\item[$\mathrm{iii}$.]
By induction, we conclude that for any $h$, $Z^{(h)}(\cdot)$ is well-defined and $a_{h}-c_{h}$ hold. In this recursive process, the sequence $\{X^{(h)}(\cdot)\}_{h\geq0}$ is  monotonically non-decreasing and bounded above, so the sequence $\{Z^{(h)}(\cdot)\}_{h\geq0}$, $\{X^{(h)}(\cdot)\}_{h\geq0}$ is convergent and $\lim_{h\rightarrow\infty}Z^{(h)}(t)=0$ for all $t \in \mathbb{R}_+$. 

Set $X^*(t)=\lim_{h\rightarrow\infty}X^{(h)}(t) \preceq \tilde{Y}_{KW}(t)$ for all $t \in \mathbb{R}_+$. Given that 
\begin{equation}
    \begin{aligned}
    &\frac{d}{dt}X^*(t)+\mathcal{G}(t,X^*_t)=\lim_{h\rightarrow\infty}\frac{d}{dt}X^{(h)}(t)+\mathcal{G}(t,\lim_{h\rightarrow\infty}X^{(h)}_t)\\
    &=-\lim_{h\rightarrow\infty}V_{(l)}^\top(t) \mathbb{R}_{22}^\sharp\left( t, X_t^{(l)} + Z_t^{(l)} \right)^{-1} V_{(l)}(t)=0,\forall t \in \mathbb{R}_+
    \end{aligned}
\end{equation}
it follows that $X^*(\cdot)$ is the $\theta$-periodic solution to SGTRDE \eqref{pf:gtrde}.
Let $\tilde{X}(\cdot)$ is the $\theta$-periodic and stabilizing solution to the SGTRDE \eqref{pf:gtrde}. From the Theorem 2 and Theorem 3 in \cite{Dragan2020}, we have $X^*(t) \succeq \tilde{X}(t)$ for all $t \in \mathbb{R}_+$.

Set \[\tilde{K}(t)=-(R_{22}(t)+\sum_{k=1}^{r}B^{\top}_{k2}(t)\tilde{X}(t)B_{k2}(t) )^{-1}(B_{02}^{\top}(t)\tilde{X}(t) + \sum_{k=1}^{r}B_{k2}^{\top}(t)\tilde{X}(t)A_k(t))\] and \[\tilde{W}(t)=-(R_{22}(t)+\sum_{k=1}^{r}B^{\top}_{k2}(t)\tilde{X}(t)B_{k2}(t))^{-1}(R_{21}(t)+\sum_{k=1}^{r}B^{\top}_{k2}(t)\tilde{X}(t)B_{k1}(t))\] and $\tilde{X}_{\tilde{KW}}(\cdot)$ is the $\theta$-periodic and stabilizing solution to \eqnref{pr:sgrde_kw} associated with $(\tilde{K}(\cdot), \tilde{W}(\cdot))$. 
By Theorem 5.6.5 in \cite{Dragan2013book}, there exists at most one stabilizing solution to the \eqnref{pr:sgrde_kw}. We thus immediately conclude that $\tilde{X}_{\tilde{KW}}(t)=X^*(t)$ for all $t \in \mathbb{R}_+$. Combining Corollary 5.6.7 from the same reference, we have $X^*(t) \preceq \tilde{X}_{\tilde{KW}}(t)=\tilde{X}(t)$ for all $t \in \mathbb{R}_+$. Therefore, $\lim_{h\rightarrow\infty}X^{(h)}(t)= \tilde{X}(\cdot)$ is the $\theta$-periodic and stabilizing solution to the SGTRDE \eqref{pf:gtrde}. 
\end{proof}

\section{Numerical Experiments}
\label{sec:numerical_experiments}

Since problems with periodic coefficients can be decomposed into a series of subproblems with deterministic coefficients, we conduct extensive numerical experiments on randomly generated system parameters across various dimensions to comprehensively validate the effectiveness and robustness of the proposed algorithm for periodic-coefficient systems and problems of varying dimensions. The key parameter settings are as follows:

\begin{table}[h]
\centering
\caption{Summary of experimental parameters}
\label{tab:param_summary}
\begin{tabular}{@{}p{5cm}p{10cm}@{}}
\toprule
Parameter & Setting \\
\midrule
Convergence tolerance & $1 \times 10^{-8}$ \\
Range of system dimensions ($n$) & From 1 to 20 \\
Trials per dimension & Fixed at 1,000 \\
Total number of trials & 20,000 (20 dimensions $\times$ 1,000 trials) \\
$n \times n$ matrices $A_0$, $A_1$, $A_2$ & Randomly generated with elements following the standard normal distribution $\mathcal{N}(0, 1)$ \\
$n \times n$ matrices $B_{01}$, $B_{02}$ & $B_{01} = 3I_{n \times m_1} - 0.5H_1$ and $B_{02} = 7I_{n \times m_2} + 0.5H_2$, where $H_1$ and $H_2$ are random matrices with entries uniformly distributed over $[0, 1]$ \\
$n \times n$ matrices $B_{11}$, $B_{12}$, $B_{21}$, $B_{22}$ & $B_{11}$, $B_{12}$, $B_{21}$, $B_{22}$ are random matrices with entries uniformly distributed over $[0, 0.01]$ \\
$n \times n$ matrix $R_{11}$ & $-4I_{m_1 \times m_1} - U_{11}^{\top}U_{11}$, where $U_{11}$ is a random matrix with entries uniformly distributed over $[0, 1]$ \\
$n \times n$ matrix  $R_{22}$ & $5I_{m_1 \times m_1} + U_{22}^{\top}U_{22}$, where $U_{22}$ is a random matrix with entries uniformly distributed over $[0, 1]$ \\
Matrices $R_{12},R_{21}, L$ & Randomly generated with elements following the uniform distribution over $[0, 1]$ \\
Matrix $M$ & Calculated as $M = U^\top U + 0.1*I_{n \times n} + L \begin{bmatrix}R_{11} & R_{12} \\ R_{21} & R_{22} \end{bmatrix}^{-1} L^{\top}$, where $U$ is random matrices with entries generated with elements following the standard normal distribution $\mathcal{N}(0, 0.1)$ \\
\bottomrule
\end{tabular}
\end{table}

\begin{figure}[htbp]
    \centering
    \includegraphics[width=1\textwidth]{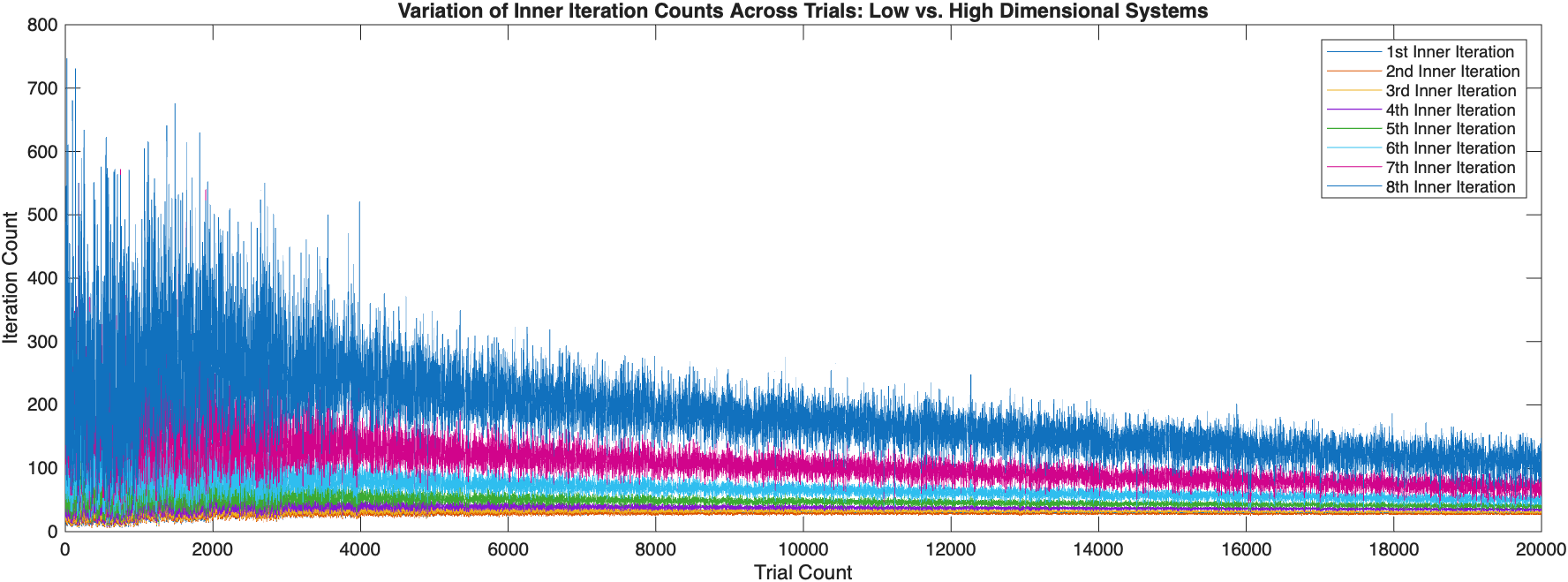} 
    \caption{Inner iteration counts for the first 8 outer iterations, ordered by computation sequence. The plot illustrates the progressive increase in inner iterations as convergence to the stabilizing solution, along with reduced variability of iterations in high-dimensional systems.}
    \label{fig:all_iterations}
\end{figure}

Under MATLAB's default random seed configuration, we performed 1000 experiments for each system with dimensions spanning 1 to 20. Across all 20,000 trials, the final results show distinct convergence patterns by outer iteration count: 266 experiments converged after 8 outer iterations, 1452 after 9, 7822 after 10, 18293 after 11, 19967 after 12, and 19998 after 13. Only 2 experiments required 15 outer iterations to converge. 
Figure \ref{fig:all_iterations} visually illustrates the number of inner iterations per round for the first 8 outer iterations. As observed from the figure, within a single computational instance, the number of inner iterations gradually increases as the solution approaches the stable solution. Additionally, the iteration count exhibits greater variability for low-dimensional systems, while fluctuations in iteration numbers diminish significantly for high-dimensional systems.

\clearpage
\printbibliography

@book{Sun2020book,
  author    = {Sun, Jingrui
               and Yong, Jiongmin},
  title     = {Linear-Quadratic Two-Person Differential Games},
  booktitle = {Stochastic Linear-Quadratic Optimal Control Theory: Differential Games and Mean-Field Problems},
  year      = {2020},
  publisher = {Springer International Publishing},
  address   = {Cham},
  pages     = {15--67},
  isbn      = {978-3-030-48306-7},
  doi       = {10.1007/978-3-030-48306-7_2},
  url       = {https://doi.org/10.1007/978-3-030-48306-7_2}
}

@book{Bacsar1998book,
  title     = {Dynamic noncooperative game theory},
  author    = {Ba{\c{s}}ar, Tamer and Olsder, Geert Jan},
  year      = {1998},
  publisher = {SIAM}
}

@article{Dragan2020,
  author   = {Dragan, V. and Aberkane, S. and Morozan, T.},
  title    = {On the bounded and stabilizing solution of a generalized Riccati differential equation arising in connection with a zero-sum linear quadratic stochastic differential game},
  journal  = {Optimal Control Applications and Methods},
  volume   = {41},
  number   = {2},
  pages    = {640-667},
  doi      = {https://doi.org/10.1002/oca.2563},
  url      = {https://onlinelibrary.wiley.com/doi/abs/10.1002/oca.2563},
  eprint   = {https://onlinelibrary.wiley.com/doi/pdf/10.1002/oca.2563},
  year = {2020}
}

@article{Kleinman1968,
  author   = {Kleinman, D.},
  journal  = {IEEE Transactions on Automatic Control},
  title    = {On an iterative technique for Riccati equation computations},
  year     = {1968},
  volume   = {13},
  number   = {1},
  pages    = {114-115},
  doi      = {10.1109/TAC.1968.1098829},
  issn     = {1558-2523},
  month    = {2}
}

@article{gajic1999,
  title     = {Solution of the state-dependent noise optimal control problem in terms of Lyapunov iterations},
  author    = {Gajic, Zoran and Losada, Ricardo},
  journal   = {Automatica},
  volume    = {35},
  number    = {5},
  pages     = {951--954},
  year      = {1999},
  publisher = {Elsevier}
}

@article{guo2001,
  title     = {Iterative solution of a matrix riccati equation arising in stochastic control},
  author    = {Guo, Chun-Hua},
  journal   = {Oper. Theory Adv. Appl},
  volume    = {130},
  pages     = {209--221},
  year      = {2001},
  publisher = {Springer}
}

@article{freiling2003,
  title     = {Properties of the solutions of rational matrix difference equations},
  author    = {Freiling, Gerhard and Hochhaus, Andreas},
  journal   = {Computers \& Mathematics with Applications},
  volume    = {45},
  number    = {6-9},
  pages     = {1137--1154},
  year      = {2003},
  publisher = {Elsevier}
}

@book{damm2004,
  title     = {Rational matrix equations in stochastic control},
  author    = {Damm, Tobias},
  volume    = {297},
  year      = {2004},
  publisher = {Springer Science \& Business Media}
}

@book{abou2012matrix,
  title     = {Matrix Riccati equations in control and systems theory},
  author    = {Abou-Kandil, Hisham and Freiling, Gerhard and Ionescu, Vlad and Jank, Gerhard},
  year      = {2012},
  publisher = {Birkh{\"a}user}
}

@book{Dragan2013book,
  author = {Dragan, Vasile and Toader, Morozan and Stoica, Adrian-Mihail},
  year   = {2013},
  month  = {01},
  title  = {Mathematical Methods in Robust Control of Linear Stochastic Systems},
  isbn   = {978-1-4614-8662-6},
  doi    = {10.1007/978-1-4614-8663-3}
}

@article{Lanzon2008,
  author   = {Lanzon, Alexander and Feng, Yantao and Anderson, Brian D. O. and Rotkowitz, Michael},
  journal  = {IEEE Transactions on Automatic Control},
  title    = {Computing the Positive Stabilizing Solution to Algebraic Riccati Equations With an Indefinite Quadratic Term via a Recursive Method},
  year     = {2008},
  volume   = {53},
  number   = {10},
  pages    = {2280-2291},
  doi      = {10.1109/TAC.2008.2006108},
  issn     = {1558-2523},
  month    = {11}
}

@article{Feng2010,
  author        = {Feng, Yantao and Anderson, Brian D. O.},
  date          = {2010-01-01},
  doi           = {https://doi.org/10.1016/j.sysconle.2009.11.006},
  isbn          = {0167-6911},
  journal       = {Systems \& Control Letters},
  number        = {1},
  pages         = {50--56},
  title         = {An iterative algorithm to solve state-perturbed stochastic algebraic Riccati equations in LQ zero-sum games},
  url           = {https://www.sciencedirect.com/science/article/pii/S0167691109001406},
  volume        = {59},
  year          = {2010}
}

@article{Dragan2011,
  author        = {Dragan, Vasile and Ivanov, Ivan G.},
  date          = {2011-07-01},
  doi           = {10.1007/s11075-010-9432-7},
  id            = {Dragan2011},
  isbn          = {1572-9265},
  journal       = {Numerical Algorithms},
  number        = {3},
  pages         = {357--375},
  title         = {Computation of the stabilizing solution of game theoretic Riccati equation arising in stochastic $H_\infty$ control problems},
  url           = {https://doi.org/10.1007/s11075-010-9432-7},
  volume        = {57},
  year          = {2011}
}

@article{Dragan2015,
  author   = {Dragan, Vasile},
  journal  = {IMA Journal of Mathematical Control and Information},
  title    = {Stabilizing solution of periodic game-theoretic Riccati differential equation of stochastic control},
  year     = {2015},
  volume   = {32},
  number   = {4},
  pages    = {839-865},
  doi      = {10.1093/imamci/dnu026},
  issn     = {1471-6887},
  month    = {12}
}

@article{Dragan2017,
  author   = {Dragan, Vasile and Aberkane, Samir},
  title    = {Computing The Stabilizing Solution of a Large Class of Stochastic Game Theoretic Riccati Differential Equations: A Deterministic Approximation},
  journal  = {SIAM Journal on Control and Optimization},
  volume   = {55},
  number   = {2},
  pages    = {650-670},
  year     = {2017},
  doi      = {10.1137/15M1049038},
  url      = {https://doi.org/10.1137/15M1049038},
  eprint   = {https://doi.org/10.1137/15M1049038}
}

@article{Ivanov2018,
  author        = {Ivanov, Ivan G. and Ivanov, Ivelin G.},
  date          = {2018-02-01},
  doi           = {10.1007/s12190-017-1086-3},
  id            = {Ivanov2018},
  isbn          = {1865-2085},
  journal       = {Journal of Applied Mathematics and Computing},
  number        = {1},
  pages         = {547--559},
  title         = {The iterative solution to LQ zero-sum stochastic differential games},
  url           = {https://doi.org/10.1007/s12190-017-1086-3},
  volume        = {56},
  year          = {2018}
}

@article{Mehrmann1988,
  author   = {Mehrmann, V. and Tan, E.},
  journal  = {IEEE Transactions on Automatic Control},
  title    = {Defect correction method for the solution of algebraic Riccati equations},
  year     = {1988},
  volume   = {33},
  number   = {7},
  pages    = {695-698},
  doi      = {10.1109/9.1282},
  issn     = {1558-2523},
  month    = {7}
}

@book{halanay2012book,
  title     = {Time-Varying Discrete Linear Systems: Input-Output Operators. Riccati Equations. Disturbance Attenuation},
  author    = {Halanay, Aristide and Ionescu, Vlad},
  volume    = {68},
  year      = {2012},
  publisher = {Birkh{\"a}user}
}

@inproceedings{Drgan2004,
  title  = {A CLASS OF NONLINEAR DIFFERENTIAL EQUATIONS ON THE SPACE OF SYMMETRIC MATRICES},
  author = {Vasile Drăgan and Gerhard Freiling and Andreas Hochhaus and Toader Morozan},
  year   = {2004},
  url    = {https://api.semanticscholar.org/CorpusID:15439718}
}

@article{Dragan1994,
  author        = {Dragan, Vasile and Halanay, Aristide and Ionescu, Vlad},
  date          = {1994-06-01},
  doi           = {10.1007/BF01206411},
  id            = {Dragan1994},
  isbn          = {1420-8989},
  journal       = {Integral Equations and Operator Theory},
  number        = {2},
  pages         = {153--215},
  title         = {Infinite horizon disturbance attenuation for discrete-time systems. A Popov-Yakubovich approach},
  url           = {https://doi.org/10.1007/BF01206411},
  volume        = {19},
  year          = {1994}
}

\end{document}